\documentclass[11pt,reqno]{amsart}
\usepackage{amsmath,amsthm,amsfonts}

\usepackage[utf8]{inputenc}
\usepackage{fullpage}
\usepackage{bbm}
\usepackage{bm}
\usepackage{amssymb}
\usepackage{shuffle}
\usepackage{dirtytalk}
\usepackage{enumerate}
\usepackage{enumitem}
\usepackage{xcolor}
\usepackage{mathtools}
\usepackage{amsmath}
\usepackage{appendix}
\usepackage{comment}
\usepackage{soul}

\usepackage{amsxtra,setspace,xspace,lmodern,psfrag,color,latexsym}
\usepackage{pifont,mathrsfs,caption,microtype,accents}
\usepackage[bookmarks=true, colorlinks=true]{hyperref}
\hypersetup{urlcolor=blue,citecolor=red,linkcolor=black}
\usepackage{cancel}
\usepackage[normalem]{ulem}

\newcommand{\D}[0]{\mathcal{D}}
\newcommand{\M}[0]{\mathcal{M}}

\newcommand{\tr}[0]{\tilde{r}}
\newcommand{\teta}[0]{\tilde{\eta}}
\newcommand{\onabla}[0]{\overline{\nabla}}

\newcommand{\V}[0]{\mathcal{V}}
\newcommand{\R}[0]{\mathbb{R}}
\newcommand{\Rdminus}[0]{\mathbb{R}^d\backslash \{x\}}

\newcommand{\B}[0]{\mathcal{B}}
\renewcommand{\S}[0]{\mathcal{S}}
\renewcommand{\P}[0]{\mathcal{P}}

\newcommand{\X}[0]{\mathcal{X}}

\newcommand{\w}[0]{\omega}
\newcommand{\norm}[1]{\lVert #1 \rVert}

\newcommand{\dd}{\mathrm{d}}
\newcommand{\Rddiag}{{\R^{2d}_{\!\scriptscriptstyle\diagup}}}
\newcommand{\cMtv}{\mathcal{\M}_{\mathrm{TV}}}
\DeclareMathOperator{\TV}{{TV}}
\DeclareMathOperator{\AC}{AC}
\newcommand{\ACt}{\mathcal{AC}_T}

\DeclareMathOperator*{\esssup}{ess\,sup}
\DeclareMathOperator*{\supp}{supp}
\newcommand{\Rd}[0]{\R^d}

\newcommand{\Linfty}[0]{L_\mu^\infty(\Rd)}
\newcommand{\LinftyK}[0]{L_\mu^\infty(K)}
\newcommand{\Kepsminus}[0]{K^\varepsilon\backslash \{x\}}
\newcommand{\KminusKeps}[0]{K\backslash K^\varepsilon}
\newcommand{\Kdiag}{{K^2_{\!\scriptscriptstyle\diagup}}}
\DeclareMathOperator*{\argmax}{arg\,max}
\DeclareMathOperator*{\argmin}{arg\,min}
\newtheorem{definition}{Definition}[section]
\newtheorem{proposition}{Proposition}[section]
\newtheorem{lemma}{Lemma}[section]
\newtheorem{assumption}{Assumption}[section]
\newtheorem{corollary}{Corollary}[section]
\newtheorem{theorem}{Theorem}[section]
\newtheorem{remark}{Remark}[section]
\newtheorem{example}{Example}[section]

  {   \end{list} }
\newcounter{broj}

\numberwithin{equation}{section}

\makeatletter
\@namedef{subjclassname@2020}{\textup{2020} Mathematics Subject Classification}
\makeatother

\title{Evolution equations on co-evolving graphs: long-time behaviour and the graph continuity equation}
\author{José Antonio Carrillo \and Antonio Esposito \and L\'aszl\'o Mikol\'as}
\address{J. A. Carrillo, L. Mikol\'as -- Mathematical Institute, University of Oxford, Woodstock Road, Oxford, OX2 6GG, United Kingdom.}
\email{carrillo@maths.ox.ac.uk}
\email{laszlomikolas@gmail.com}
\address{A. Esposito -- Department of Information Engineering, Computer Science and Mathematics, Università degli Studi dell'Aquila, Via Vetoio 1, Coppito, 67100 L'Aquila, Italy.}

\email{antonio.esposito3@univaq.it}

\begin{document}

\keywords{Co-evolving graphs, evolution on graphs, long-time behaviour, nonlocal equations}
\subjclass[2020]{35R02, 35A01, 35A02, 35A24, 35B40}

%35R02 PDEs on graphs and networks (ramified or polygonal spaces)

%35R06(2010–now) PDEs with measure

%35A01(2010–now) Existence problems for PDEs: global existence, local existence, non-existence

%35A02(2010–now) Uniqueness problems for PDEs: global uniqueness, local uniqueness, non-uniqueness

%35A24(2010–now)Methods of ordinary differential equations applied to PDEs

%35B40(1973–now)Asymptotic behavior of solutions to PDEs

\begin{abstract}
We focus on evolution equations on co-evolving, infinite, graphs and establish a rigorous link with a class of nonlinear continuity equations, whose vector fields depend on the graphs considered. More precisely, weak solutions of the so-called graph-continuity equation are shown to be the push-forward of their initial datum through the flow map solving the associated characteristics' equation, which depends on the co-evolving graph considered. This connection can be used to prove contractions in a suitable distance, although the flow on the graphs requires a too limiting assumption on the overall flux. Therefore, we consider upwinding dynamics on graphs with pointwise and monotonic velocity and prove long-time convergence of the solutions towards the uniform mass distribution.
\end{abstract}

\maketitle

\section{Introduction}
In this manuscript we continue the study of evolutionary partial differential equations (PDEs) on \textit{co-evolving graphs} initiated in \cite{esposito_mikolas_2024}. There, we analysed a model describing the evolution of some quantity, e.g. mass, distributed on the vertices of a graph and flowing through its weighted edges. In this model, the dynamics of the mass are coupled with the dynamics of the edge-weights, meaning that the dynamics on the graph depends on those determining the graph itself. In many cases, the adaptive nature of \textit{co-evolving} graphs is better suited for modelling many real-life phenomena than their static or simply evolving counterparts --- when the edge weights also evolve in time, but their dynamics are not coupled with those on the graph. To mention a few examples, \textit{co-evolving} graphs have been useful in biology, modelling neuronal systems \cite{clopath2010connectivity} or physiological networks \cite{sawicki2022modeling}, in machine learning, for modelling the training process of deep neural networks \cite{triesch2005gradient}, and in epidemiological models, where individuals susceptible to some disease (the nodes) might want to break off the links (the edges) to other individuals because they are afraid of becoming infected \cite{demirel2017dynamics}. We refer the reader to the review~\cite{berner2023adaptive} for more applications of \textit{co-evolving} graphs.

In this article, we build on the setting of~\cite{Esposito_Pattachini_Schlichting_Slepcev,Esposito_on_a_class}, where the authors study the evolution of a (possibly negative) quantity on the vertices, a measure $\rho$ on a \textit{static} graph, defined as follows. 

\begin{definition}\label{def:intro_graph}
    A static graph is given by a pair $\mathcal{G}:= (\mu, \eta)$, where $\mu \in \M^{+}(\Rd)$ is a positive Radon measure whose support describes the set of vertices and $\eta \in C_b(\Rddiag)$ is the edge-weight function, continuous and bounded, on the possible set of edges $\Rddiag:= \{(x,y) \in \R^{2d} \ | \ x \neq y\}$.
\end{definition}
In view of the previous definition, self-loops are not allowed and, for example, any finite set of vertices, $\{x_1,\ldots,x_n\} \subset \R^{nd}$, corresponds to empirical measures such as $\mu^n=\frac{1}{n}\sum_{i=1}^n\delta_{x_i}$, where $\delta_{x_i}$ denotes a dirac-mass at each $x_i\in \Rd$, for $i=1,\ldots, n$. The same setting allows to easily represent large networks by taking, e.g., $\mu^\infty = \lim_{n\to \infty}\mu^n$, interpreting the limit in a suitable weak-sense. Natural examples of the edge-weight function $\eta$ are functions of the distance, that is $\eta = f(|x-y|)$. In particular, Definition~\ref{def:intro_graph} includes both finite and uncountably infinite weighted graphs.

In \cite{esposito_mikolas_2024}, some of the previous results are extended to the \textit{co-evolving} setting and it is studied the well-posedness of the following \textit{co-evolving} non-local conservation law  
\begin{equation}\label{eq:intro_ivp}
   \begin{split} 
    \partial_t \rho_t & = - \onabla \cdot F^\Phi[\mu, \eta_t ; \rho_t,V_t[\rho_t]],  
    \\
    \partial_t \eta_t & =  \w_t[\rho_t] -\eta_t, 
    \end{split}\tag{Co-NCL}
\end{equation}
where, for $t\in[0,T]$, the edge-weight $\eta_t$ is no longer static and its dynamics depend on the ``mass'' distribution on the vertices, $\rho_t$, through a (potentially non-local) function $\w$. As for the equation for $\rho$, it is a non-local continuity equation where $\onabla \cdot$ is the graph divergence (see Definition \ref{def:non_local_grad_div}), $F^\Phi$ is the flux, and $V$ is a solution dependent velocity. Differently from evolutions on $\Rd$, note that, on the graph, the mass is distributed on the vertices, whereas the velocity and the flux are defined on the edges. In particular, to define the flux we need a suitable function, $\Phi$, interpolating the mass at the vertices connected by a given edge to have edge-based quantities. For this reason, the flux depends on $\Phi$ as the notation $F^\Phi$ indicates, cf. Definition \ref{def:admissible_interpolation}. We refer the reader to Section~\ref{sec:preliminaries} for further details.

In this paper, we shall focus on the case where the measure defining the set of vertices, $\mu\in \P_2(\Rd)$, is a probability measure with finite second moments and $\rho \ll \mu$ with $r:= \frac{\dd \rho}{\dd \mu}$. Under suitable assumptions on $\Phi$ and $\eta$, system \eqref{eq:intro_ivp} can be rewritten as, for $\mu$-a.e. $x\in\Rd$,
\begin{equation}\label{eq:intro_euler}
   \begin{split} 
    \partial_t r_t(x) & = - \int_{\Rd\backslash\{x\}}\Phi(r_t(x),r_t(y); V_t[r_t](x,y))\eta_t(x,y)\dd\mu(y),
    \\
    \partial_t \eta_t & =  \w_t[r_t] -\eta_t,
    \end{split}\tag{Euler Co-NCL}
\end{equation}
where the equation for $r$ is a nonlinear Euler equation, motivating the label chosen. Starting from~\eqref{eq:intro_euler}, first we prove existence and uniqueness of solutions to~\eqref{eq:intro_euler} as curves in the space $C([0,T], {L^p_\mu(\Rd)})\times C([0,T], C_b(\Rddiag))$, $p=2,\infty$, equipped with a suitable distance we specify later in Section~\ref{sec:euler-co-ncl}. This extends the result in~\cite{esposito_mikolas_2024} to the $L^p_\mu$ setting, which is needed to study the two main aspects of this paper: the associated graph continuity equation and the long-time behaviour of its solutions. First, we shall see that, for $\rho\ll\mu$, we can link~\eqref{eq:intro_ivp} to a particular nonlinear continuity equation on $\R^{d+1}$ which we shall call the \textit{graph continuity equation}. More precisely, for a given pair $(r, \eta)$ solution of~\eqref{eq:euler} with base measure $\mu \in \P_2(\Rd)$, the measure $\sigma:= \delta_r \otimes \mu$ defined by 
\begin{equation}\label{eq:intro_mu_monokinetic}
   ( \delta_r\otimes \mu)(A \times B) = \int_{B}\delta_{r(x)}(A) \dd \mu(x), 
\end{equation}
on any Borel set $A\times B \in \B(\R\times\Rd)$, is a weak solution to the following continuity equation 
\begin{equation}\label{eq:intro_vlasov_equation}
\begin{cases}
  \partial_t \sigma + \partial_{\xi}(\sigma\X[\sigma, r, \eta]) = 0,
    \\
     \sigma_0 = \bar{\sigma} \in \P(\R \times \Rd),
\end{cases}\tag{graph-CE}
\end{equation}
where, for any $\sigma \in \P(\R\times\Rd)$, the velocity field is given by
\begin{equation}\label{eq:intro_mean_field}
\X[\sigma,r, \eta](t,\xi,x):=-\int_{\R \times \Rdminus}\Phi( \xi, \xi'; V_t[r_t](x,x'))\eta_t(x,x')\dd \sigma(\xi',x'), 
\end{equation}
and it depends on a solution $(r,\eta)$ to~$\eqref{eq:intro_euler}$.  We refer to~\eqref{eq:intro_vlasov_equation} as the \textit{graph-continuity equation} since the velocity field depends, indeed, on the solution of the co-evolving graph problem. By looking at the disintegrated version of~\eqref{eq:intro_vlasov_equation}, we prove that weak solutions can be constructed as the push-forward of the initial datum through the corresponding flow map. We embed~\eqref{eq:intro_euler} into~\eqref{eq:intro_vlasov_equation} in order to study the long-time asymptotics via contractions in a suitable distance, as in Section~\ref{sec:contraction}. Obtaining contractivity of the flow is, indeed, difficult and still open for solutions to~\eqref{eq:intro_ivp} in view of the non-standard possible Finslerian nature of the equation for $\rho$, cf.~\cite{Esposito_Pattachini_Schlichting_Slepcev}. Working with~\eqref{eq:intro_vlasov_equation}, we obtain contraction of the solutions, although the structure of the flow somehow requires a too strong condition on the overall flux which does not seem to be satisfactory, since it indicates that all vertices should loose mass, which is at odds with most examples in practice. While we can prove a theoretical result, we believe a more reasonable condition should be provided. With the aim of obtaining a better result in this direction, the second contribution of this manuscript is the study of the long-time behaviour of \eqref{eq:intro_euler} for the upwind interpolation and a particular type of velocity fields that we call \textit{pointwise} --- they can be written as $V_t[r](x,x') = V_t(r_t(x),r_t(x'))$, for any vertices $x,x' \in \Rd$ --- and \textit{monotonic} (see Definition \ref{def:monotonic_velocity}), which intuitively means that vertices will send mass to neighbours with less mass and gain mass from those with more mass. In view of the pointwise nature of the velocity fields, we obtain the well posedness of~\eqref{eq:intro_euler} for $r \in L^\infty_\mu(\Rd)$ by the Banach-fixed point theorem in Theorem \ref{thm:well_posedness_euler_infty}. We then move to the study of the long-time behaviour of \eqref{eq:intro_euler} in this framework. To this end, we restrict the co-evolving graph to be defined on a compact set $K \subset \Rd$ and fix the upwind interpolation $\Phi_{upwind}$. We obtain the long-time asymptotics by showing that the supremum norm of $r$ decreases monotonically in time, while the infimum increases montonically in time, resulting in a stabilisation of the dynamics at $\frac{M}{\mu(K)}$, the uniform mass distribution, where $\int_{\Rd}r_0(x)\dd\mu(x) = M$. As byproduct, we provide long-time asymptotics for the $\mu$-monokinetic solution of~\eqref{eq:intro_vlasov_equation}.

\medskip

Passing from~\eqref{eq:intro_euler}~to~\eqref{eq:intro_vlasov_equation} resembles the link between the graph-limit and the mean-field limit for interacting particle systems (IPSs) with labels, cf., e.g.,~\cite{paul2024microscopic} and the references therein. This approach, in a different framework with respect to the one presented in this paper, concerns IPSs with heterogeneous interactions, depending on the particles' identities. Thus, unlike traditional \textit{exchangeable} IPSs where the system is completely invariant under any permutation of the variables' labels, particles are no longer exchangeable. In this regard, a powerful way to keep track of the identities of the particles and their connections is to use weighted graphs. More precisely, the labels can be encoded into a vertex set $V:=\{1,\ldots,N\}$ and the heterogeneous interactions can be modelled by a family of interaction potentials $(W^{i,j})_{i,j=1}^N$ depending on which pair of particles $(i,j) \in E$ is being considered, where $E = V \times V$ is the set of edges. For example, one can consider the following IPS composed of particles $X_t^i \in \Rd$, with labels $i=1,\ldots, N$, evolving according to
\begin{equation}\begin{cases}\label{eq:agg_eq_particles_graph}
        \frac{d}{dt}X^i_t = -\frac{1}{N}\sum_{i\neq j}^N\kappa({i,j})\nabla f(X_t^i-X_t^j),
        \\
        X_0^i \in \Rd, \tag{Het-IPS}
\end{cases}
\end{equation}
where the family of potentials $(W^{i,j})_{i,j=1}^N$ is given by $W^{i,j}:=\kappa({i,j})f(X_t^i-X_t^j)$, and 
$\kappa:E \to \R$ encodes the differences in the interactions between the particles depending on their labels.
Therefore, the vertices of this graph can be viewed as a collection of point masses that move in space, with the trajectory of each point mass being influenced by the rest. We refer to this approach as the \textit{fixed-mass} approach, in contrast with the framework of this paper, where the positions of the masses are fixed, but the mass is allowed to evolve. We refer to the latter point of view as the \textit{evolving-mass} approach, which is useful, for instance, in applications to data science and machine learning, since data have usually the form of point clouds, which constitute the sets of vertices. These vertices are, indeed, fixed and the weighted edges can encode some measure of similarity, such as the distance between these points. In this context, graphs are a powerful tool to represent the data. In particular, dynamics on graphs can be used to detect meaningful clusters, as explained in~\cite{craig2021clustering}. Let us stress that in this context the position is fixed (vertices) whereas the ``mass'' is moving.

The study of IPSs on heterogeneous graphs, i.e. not all-to-all with homogeneous weights, can be traced back to the series of seminal papers by Medvedev~\cite{Medv14_SIAM,Medv14}, studying a nonlinear heat-equation on graphs and its continuum limit as the number of vertices (particles) goes to infinity. To the best of the authors' knowledge, these contributions were the first to rigorously justify the use of continuum evolution equations, called the \textit{graph limit}, to approximate dynamics on large dense graphs by using \textit{graphons}, a notion of limiting graph for sequences of dense graphs~\cite{lovasz2012large}. The graph limit describes the trajectory of an agent with a given label when the number of agents goes to infinity. The notion of graph limit in the context of IPSs has been studied extensively; we mention, e.g.,~\cite{ayi2023graph,Gkogkas_Kuehn_Xu_CMS,porat2023mean,el2023continuum,favre2024continuum,haskovec2024graph,BP_pairwise_competition}. Continuity equations similar to~\eqref{eq:intro_vlasov_equation} have been obtained in the study of the mean-field limit of IPSs with heterogeneous interactions, see for example~\cite{Kaliuzhnyi,ChibaI,ChibaII, Kuehn_graphops1, Gkogkas_graphop,Vlasov_digraph} for static graphs, although this list is not exhaustive. In the case of co-evolving graphs, we highlight the recent work~\cite{gkogkas2022mean_heterogeneous} on the Kuramoto model on a co-evolving graph, where the authors prove the mean-field limit and provide several examples both for dense as well as sparse graphs. These works rely on the recent theory of graph limits~\cite{lovasz2012large,Backhausz}. We also mention~\cite{ayi2024mean,paul2024microscopic} for recent reviews on the different approaches to defining the mean-field limit of heterogeneous IPS. In this sense, the first contribution in this manuscript can be viewed as an application of the ideas in ~\cite[Section 5.2.3]{paul2024microscopic} on how to obtain a continuity equation starting from the solution of a nonlinear Euler equation, although our setting is not included in~\cite{paul2024microscopic}. A common point in many of the aforementioned works and the one presented in this manuscript is the approach to obtain the well-posedness of the nonlinear continuity equation. Namely, the well-posedness is obtained by studying a family of characteristic equations parametrized by the vertices of the underlying graph extending the work of Neunzert,~\cite{Neunzert}. To the best of the author's knowledge, the first application of this approach in the context of all-to-all graphs can be found in~\cite{Lancellotti} and, in the case of heterogeneous graphs can be found in~\cite{Kaliuzhnyi}.

The approach we use in Section~\ref{sec:long_time_behaviour} to study the long-time behaviour for~\eqref{eq:intro_euler} follows similar ideas to the ones in~\cite{Nastassia_Long_Time}, where the authors study the long-time behaviour of large systems of interacting agents on time-evolving graphs.
Indeed, our result is linked to the problem of \textit{consensus formation} for systems of \textit{non-exchangeable} interacting particles. More precisely, this problem consists in determining whether there exists an element $x^\infty$ such that 
\[
\lim_{t \to \infty} |x_i(t) - x^\infty|= 0,
\]
for any agent $i \in \{1,\ldots,N\}$ in an IPS. In our case, the evolving-mass approach, we obtain that $x^\infty = \frac{M}{\mu(K)}$, the uniform mass distribution. Consensus formation has received a vast amount of attention, see for example \cite{Trelat_long_time, Trelat_Cucker_Smale_long_time, Ernesto_CS, Moreau_rev, Olfati-Saber} and the references therein. The literature on this topic is very extensive as techniques have to be adapted to the particular type of interaction and connectivity in each case. 
Our approach is close to the works that tackle the problem by obtaining \textit{diameter} decay estimates.  We obtain that the rate at which the \say{vertex with maximum mass} (\say{vertex with minimum mass}) loses (gains) mass depends on the connectivity of the system through the quantities $\eta_*$, the smallest connectivity for any vertex in the history of the co-evolving graph, and $\alpha'_*$, the smallest rate of outflow for the velocities. This is closely related to the \textit{scrambling-coefficient},~\cite{seneta1979coefficients}, obtained for example in~\cite{Nastassia_Long_Time}, a quantity which is positive if either a pair of agents is interacting or they follow a common third agent. Our condition is stronger, since we impose that the graph is strongly connected, i.e., $\eta_t >0$ for $t\geq 0$. This is due to the co-evolving nature of $\eta$ which is not present in most of the aforementioned works, where the connectivity of the topology is, at most, allowed to evolve in time, but not to depend on the dynamics of the agents, as in our case. Finding weaker conditions on $\eta$ yielding convergence to consensus could be achieved upon investigating different co-evolving dynamics. 

\medskip
The setting outlined in Definition \ref{def:intro_graph} was introduced in~\cite{Esposito_Pattachini_Schlichting_Slepcev} where the authors considered nonlocal dynamics on graphs driven by nonlocal interaction energies. Similar frameworks can be also found in~\cite{Garcia_Trillos_Slepcev_Continuum_Limit,Garcia_Trillos_Slepcev_variational} in static problems. The work~\cite{Esposito_Pattachini_Schlichting_Slepcev} was extended in different directions. In~\cite{Esposito_on_a_class}, the authors studied a generalisation of the dynamics by looking at a wider class of interpolation functions $\Phi$ (the ones considered in this paper), while~\cite{Esposito_Pattachini_Schlichting_Slepcev} focused on the upwind interpolation $\Phi_{upwind}$ (see Example \ref{ex:example_interpolations}), also studied in~\cite{ChTrTa18}. In~\cite{Heinze_Nonlocal_cross_interaction_systems_on_graphs}, the authors extended the framework of~\cite{Esposito_Pattachini_Schlichting_Slepcev} to the case of two interacting species on graphs. They allow for non-linear mobilities and consider a suitable modified $p$-Wasserstein quasi-metric for $p\in (1,\infty)$, generalising the metric defined in~\cite{Esposito_Pattachini_Schlichting_Slepcev}. The authors show that the multi-species dynamics also define a gradient flow of the corresponding interaction energy with respect to the generalised version of $p$-Wasserstein (quasi) metric on graphs. The study of nonlocal cross-interaction equations on graphs is continued by the same authors in~\cite{Heinze_energy_landscape}, where they investigate certain qualitative properties of solutions for finite graphs. Namely, the authors give conditions on the self-interaction and cross-interaction kernels leading to energy minimisers and give conditions for obtaining stationary states. Finally, they study the stability of stationary states, the formation of patterns and the effect of the connectivity $\eta$ on the resulting dynamics. They carry out these efforts analytically for small graphs of up to four vertices and present numerical results for graphs of up to one hundred vertices. In~\cite{esposito2023graphtolocal}, the authors studied a natural question posed in~\cite{Esposito_Pattachini_Schlichting_Slepcev}: does one recover the nonlocal interaction equation on $\Rd$ as the graph localises? In other words, as the range of connections between different vertices decreases, but the weight of each connecting edge increases, do the nonlocal dynamics on the graph converge to the dynamics on Euclidean space? In \cite{esposito2023graphtolocal}, the authors show that, as the graph localises, the obtained dynamics is the nonlocal interaction equation on Euclidean space, with a tensor mobility $\mathbb{T}$, depending on the geometry of the graph through the connectivity, $\eta$, and the base measure, $\mu$. Furthermore, the resulting PDE is a Riemannian gradient flow on $(\P_2(\Rd_{\mathbb{T}}),d_2)$, where $\Rd_{\mathbb{T}}$ is the $d$-dimensional Euclidean space endowed with a metric induced by $\mathbb{T}^{-1}$. This result establishes a link between the Finslerian gradient flow structure obtained in~\cite{Esposito_Pattachini_Schlichting_Slepcev} and the Riemannian gradient flow obtained in the limiting Euclidean case; see~\cite{Esposito_Graph_to_local_multispecies} for an extension to systems of $N$ nonlocal cross-interaction equations on graphs. In particular, graphs are space-discretisation alternative to tessellations. In this regard, for interesting discrete-to-continuum evolution problems, we refer the reader to~\cite{DisserLiero2015,forkertEvolutionaryGammaConvergence2020,hraivoronska2023diffusive,HraivoronskaSchlichtingTse2023}. Furthermore, we also mention~\cite{GLADBACH2020204,GladbachKopferMaasPortinale2023} for discrete optimal transport distances and their convergence to their continuous counterparts. Manuscripts dealing with evolutions on finite graphs are~\cite{maas2011gradient,Chow_Huang_Li,Mielke2011gradient}, where the concept of Wasserstein metric on finite graphs was introduced independently. Related nonlocal Wasserstein distances are studied in~\cite{SlepcevWarren2022}, as well as the recent~\cite{Warren_preprint_25} which focuses on a class of nonlocal diffusion equations driven by relative entropy. We also mention~\cite{Bianchi_PME_Graphs_CVPDE}, where the authors provide a well-posedness theory for generalised porous medium equation on infinite graphs.

\subsection*{Structure of the manuscript}
The rest of the manuscript is structured as follows. In section~\ref{sec:preliminaries} we recall relevant definitions and results from~\cite{Esposito_Pattachini_Schlichting_Slepcev,Esposito_on_a_class,esposito_mikolas_2024}.
Focusing on the case $\rho\ll\mu$, in section~\ref{sec:euler-co-ncl} we study the well-posedness of~\eqref{eq:intro_euler}, derived from~\eqref{eq:intro_ivp}. In Section \ref{sec:graph_CE}, we show how we obtain \eqref{eq:intro_vlasov_equation} from \eqref{eq:intro_euler}, prove its well-posedness and a stability estimate. We also discuss how this equation was obtained in an attempt to study the long-time behaviour of \eqref{eq:intro_ivp}. Considering the upwind interpolation, section~ \ref{sec:long_time_behaviour} focuses on the long-time behaviour of~\eqref{eq:intro_euler}, and~\eqref{eq:vlasov_equation} as consequence, for pointwise monotonic velocities.

\section{Preliminaries}\label{sec:preliminaries}

We shall focus on \textit{co-evolving} graphs, denoted by $\mathcal{G}_t := (\mu, \eta_t[\rho])$, in which the weight function is allowed to change in time depending on the mass configuration, $\rho$, defined on the vertices. 

Since it is useful in our analysis, we recall the notation and the results obtained in~\cite{Esposito_on_a_class, esposito_mikolas_2024}. There, the unknown $\rho \in \cMtv(\Rd)$ is a signed Radon measure with finite total variation and the set $\cMtv$ is equipped with the total variation norm, defined using the dual product between $C_0(\R^d)$ and $\mathcal{M}(\R^d)$. Furthermore, since the continuity equation induces mass preserving dynamics that do not increase the total variation, we have $\rho$ belongs to $\cMtv^M(\R^d):=\left\{\rho\in\cMtv(\R^d):|\rho|(\R^d)\leq M\right\}$. We shall deal with dynamics on a non-Euclidean setting, hence we recall the notion of gradient and divergence on co-evolving graphs, as in~\cite{Esposito_Pattachini_Schlichting_Slepcev}.
\begin{definition}[Nonlocal gradient and divergence]\label{def:non_local_grad_div}
For any $\phi: \mathbb{R}^d \rightarrow \mathbb{R}$, we define its nonlocal gradient ${\onabla} \phi: \Rddiag \rightarrow \mathbb{R}$ by
$$
{\onabla} \phi(x, y)=\phi(y)-\phi(x), \quad \text { for all }(x, y) \in \Rddiag .
$$
For any Radon measure $\boldsymbol{j} \in \mathcal{M}(\Rddiag)$, its nonlocal divergence ${\onabla} \cdot \boldsymbol{j} \in \mathcal{M}(\mathbb{R}^d)$ is defined as the adjoint of ${\onabla}$, i.e., for any $\phi: \mathbb{R}^d \rightarrow \mathbb{R}$ in $C_0(\R^d)$, there holds
$$
\begin{aligned}
\int_{\mathbb{R}^d} \phi \mathrm{d} {\onabla} \cdot \boldsymbol{j} & =-\frac{1}{2} \iint_\Rddiag {\onabla} \phi(x, y)\mathrm{d} \boldsymbol{j}(x, y) \\
& =\frac{1}{2} \int_{\mathbb{R}^d} \phi(x) \int_{\mathbb{R}^d \backslash\{x\}} (\mathrm{d} \boldsymbol{j}(x, y)-\mathrm{d} \boldsymbol{j}(y, x)) .
\end{aligned}
$$
In particular, for $\boldsymbol{j}$ antisymmetric, that is, $\boldsymbol{j} \in \mathcal{M}(\Rddiag)$ and $\dd\boldsymbol{j}(x,y)=-\dd\boldsymbol{j}(y,x)$, denoted $\boldsymbol{j} \in$ $\mathcal{M}^{\mathrm{as}}(\Rddiag)$, we have
$$
\int_{\mathbb{R}^d} \phi \mathrm{d} {\onabla} \cdot \boldsymbol{j}=\iint_\Rddiag \phi(x)\mathrm{d} \boldsymbol{j}(x, y) \ .
$$
\end{definition}
The solutions to the nonlocal continuity equations are curves on a time interval $[0,T]$, with $T>0$, and we denote by $\AC([0,T];\cMtv(\R^d))$ the set of curves from $[0,T]$ to $\cMtv(\R^d)$ absolutely continuous with respect to the $\TV$ norm, that is, the set of curves $\rho\colon [0,T] \to \cMtv(\R^d)$
such that there exists $m\in L^1([0,T])$ with
\begin{equation*}
  \norm{\rho_t-\rho_s}_{\TV} \leq \int_s^t m(r) \dd r, \qquad \text{for all } 0\leq s< t\leq T. 
\end{equation*} 
The definition of the flux requires further considerations in the graph setting since the mass is a vertex based quantity, while the velocity is an edge based quantity. It is necessary to interpolate the mass at two given adjacent vertices to create an edge based quantity according to an admissible flux interpolation. The definition of admissible flux interpolations and fluxes follows~\cite{Esposito_on_a_class,esposito_mikolas_2024}. 

\begin{definition}[Admissible flux interpolation]\label{def:admissible_interpolation}  A measurable function $\Phi: \mathbb{R}^3 \rightarrow \mathbb{R}$ is called an admissible flux interpolation provided that the following conditions hold:
\begin{enumerate}[label=(\roman*)]
    \item\label{ass:interp_deg} $\Phi$ satisfies
$$
\Phi(0,0 ; v)=\Phi(a, b ; 0)=0, \quad \text { for all } a, b, v \in \mathbb{R} ;
$$
 \item\label{ass:interp_lip} $\Phi$ is argument-wise Lipschitz in the sense that, for some $L_{\Phi}>0$, any a, $b, c, d, v, w \in$ $\mathbb{R}$, it holds
$$
\begin{aligned}
|\Phi(a, b ; w)-\Phi(a, b ; v)| & \leq L_{\Phi}(|a|+|b|)|w-v|; 
\\
|\Phi(a, b ; v)-\Phi(c, d ; v)| & \leq L_{\Phi}(|a-c|+|b-d|)|v|;
\end{aligned}
$$
\item $\Phi$ is positively one-homogeneous in its first and second arguments, that is, for all $\alpha>0$ and $(a, b, w) \in \mathbb{R}^3$, it holds
$$
\Phi(\alpha a, \alpha b ; w)=\alpha \Phi(a, b ; w).
$$
\end{enumerate}
\end{definition}

\begin{example}\label{ex:example_interpolations}
A particularly useful interpolation, especially for potential vector fields, is the upwind interpolation, used, e.g., in~\cite{Esposito_Pattachini_Schlichting_Slepcev},
\[
\Phi_{\text {upwind }}(a, b ; w)=a w_{+}-b w_{-}, \qquad \mbox{ for } (a,b,w)\in\R^3.
\]
Mean multipliers represent another example, given by
\[
\Phi_{\text {prod }}(a, b ; w)=\phi(a, b)w, \qquad \mbox{ for } (a,b,w)\in\R^3,
\]
where common choices for $\phi$ include: $\phi(a,b) = \frac{a+b}{2}$, or $\phi(a,b) = \max\{a,b\}$.
\end{example}

An admissible flux is defined as follows.

\begin{definition}[Admissible flux]\label{def:flux} Let $\Phi$ be an admissible flux interpolation, and let $\rho \in \mathcal{M}_{\mathrm{TV}}(\mathbb{R}^d)$, $w \in \mathcal{V}^{\mathrm{as}}(\Rddiag):=\left\{v: \Rddiag \rightarrow \mathbb{R}| v(x,y)=-v(y,x)\right\}$, and $\eta:\Rddiag\to \R$ measurable. Furthermore, take {a reference measure} $\lambda \in \mathcal{M}^+(\mathbb{R}^{2 d})$ such that $\rho \otimes \mu, \mu \otimes \rho \ll \lambda$. Then, the admissible flux $F^{\Phi}[\mu, \eta ; \rho, w] \in \mathcal{M}(\Rddiag)$ at $(\rho, w)$ is defined by 
$$
\mathrm{d} F^{\Phi}[\mu, \eta ; \rho, w]=\Phi\left(\frac{\mathrm{d}(\rho \otimes \mu)}{\mathrm{d} \lambda}, \frac{\mathrm{d}(\mu \otimes \rho)}{\mathrm{d} \lambda} ; w\right) \eta\, \mathrm{d} \lambda .
$$
\end{definition}
The one-homogeneity of $\Phi$ implies the definition does not depend of the choice of $\lambda$, as long as the absolute continuity requirement is considered, cf.~\cite[Remark 2.6]{Esposito_on_a_class}. The \textit{co-evolving graphs}, $\mathcal{G}_t$, are determined by the following initial value problem introduced in~\cite{Esposito_on_a_class}
\begin{equation}\label{eq:ivp}
   \begin{split} 
    \partial_t \rho_t & = - \onabla \cdot F^\Phi[\mu, \eta_t ; \rho_t,V_t[\rho_t]],  
    \\
    \partial_t \eta_t & =  \w_t[\rho_t] -\eta_t, 
    \end{split}\tag{Co-NCL}
\end{equation}
for given initial data $\rho^0 \in \cMtv^M(\R^d)$ and $\eta^0\in C_{b}(\Rddiag)$. Focusing on the dynamics for $\eta$, we consider $\w:[0,T]\times \cMtv(\Rd) \times \Rddiag \to \R$ satisfying the following assumptions: 
\begin{enumerate}[label=$(\bm{\w}\arabic*)$]
    \item the map $(x,y)\in\Rddiag  \mapsto \w_t[\cdot](\cdot,x,y)$ is continuous and $(t\mapsto\w_t[\cdot](\cdot,\cdot))\in L^1([0,T])$;\label{ass:w_continuous}
    \item for any $\rho, \sigma \in \cMtv(\R^d)$ there exists a constant $L_\w\geq 0$ such that \label{ass:omega_lip}
     \begin{align*}
         \sup_{t \in [0,T]}\sup_{x,y \in \Rddiag} | \w_t[\sigma](x,y) - \w_t[\rho](x,y)| & \leq L_\w\norm{\sigma - \rho}_{TV}; 
     \end{align*}
     \item $\w$ is bounded, that is, there exists a constant $C_\w>0$ such that \label{ass:omega_bounded}
     \begin{equation*}
    \sup_{t \in [0,T]}\sup_{\rho \in \cMtv(\R^d)}\sup_{x,y \in \Rddiag}\big|\w_t[\rho](x,y)\big| \le C_\w.
\end{equation*}
\item The map $(x,y) \mapsto \w_\cdot[\cdot](x,y)$ is symmetric. \label{ass:omega_symmetric}
\end{enumerate}
We observe that assumption~\ref{ass:omega_symmetric} is not required in~\cite{esposito_mikolas_2024} since it is not needed for the well posedness of~\eqref{eq:ivp}. However, in what follows, we will require $\eta$ to be symmetric, hence~\ref{ass:omega_symmetric} is necessary.
A solution to~\eqref{eq:ivp} is understood as follows. 

\begin{definition}[Solution to the initial value problem \eqref{eq:ivp}]\label{def:sol_to_ivp}
   Given an admissible flux interpolation $\Phi$, a velocity field $V:[0,T]\times \cMtv(\R^d)\to \mathcal{V}^{as}(\Rddiag)$, and function $\w:[0,T]\times \cMtv(\R^d)\times \Rddiag \to \R$, a pair $(\rho, \eta): [0,T] \to \cMtv(\R^d)\times  C_{b}(\Rddiag)$ is a solution to the initial value problem \eqref{eq:ivp} if, for any $\varphi \in C_0(\Rd)$, 
   \begin{enumerate}[label=(\roman*)]
       \item $\rho \in AC([0,T], \cMtv(\R^d)),\ \eta \in AC([0,T],  C_{b}(\Rddiag))$; \label{cond:sol_1}
       \item the maps $t \mapsto \langle \varphi,\onabla \cdot F^\Phi[\mu, \eta_t ; \rho_t,V_t[\rho_t]]\rangle$ and $t \mapsto  \w_t[\rho_t] -\eta_t$ belong to $L^1([0,T])$;  \label{cond:sol_2}
       \item for a.e. $t \in [0,T]$, every $(x,y)\in\Rddiag$,  for any $\varphi\in C_0(\R^d)$, the following conditions hold \label{cond:sol_3}
       \begin{align}
           \int_{\R^d}\varphi(x)\dd \rho_t(x)  & = \int_{\Rd}\varphi(x)\dd \rho_0(x) + \frac{1}{2} \int_{0}^t \iint_{\Rddiag}\onabla\varphi(x,y)\dd F^\Phi[\mu,\eta_s,\rho_s;V_s[\rho_s]](x,y) \dd s\label{eq:rho_evolution}\ ,
           \\
           \eta_t(x,y) &= \eta_0(x,y)  + \int_0^t  \left(\w_s[\rho_s](x,y) -\eta_s(x,y)\right) \,\dd s.
           \label{eq:eta_evolution}
       \end{align}
   \end{enumerate}
\end{definition}

We recall the main results needed to ensure well-posedness of \eqref{eq:ivp} from \cite{esposito_mikolas_2024}. 

\begin{theorem}[Existence and uniqueness for~\eqref{eq:ivp}]\label{thm:well-posedness} Let $V:[0,T]\times \cMtv^M(\R^d)\to \V^{as}(\Rddiag)$ satisfy the uniform compressibility assumption
\begin{equation}\label{eq:velocity_bound}
  \sup _{t \in[0, T]} \sup _{\rho \in \M_{TV}^M(\mathbb{R}^d)} \sup _{x \in \mathbb{R}^d} \int_{\mathbb{R}^d \backslash\{x\}}\left|V_t[\rho](x, y)\right| \mathrm{d} \mu(y)  \leq C_V\ ,
\end{equation}
for $C_V>0$. Assume there is a constant $L_V > 0$ such that, for all $t \in [0,T]$ and all $\rho, \sigma \in \M_{TV}^M(\R^d)$:
\begin{equation}\label{eq:ass_Lip_rho}
         \sup _{x \in \mathbb{R}^d} \int_{\mathbb{R}^d \backslash\{x\}}\left|V_t[\rho](x, y)-V_t[\sigma](x, y)\right| \mathrm{d} \mu(y)  \leq L_V\|\rho-\sigma\|_{\mathrm{TV}} \ .
\end{equation}
Let $\w:[0,T]\times\cMtv(\Rd)\times\Rddiag\to\R$ satisfy~\ref{ass:w_continuous}---\ref{ass:omega_bounded}. Then there exists a unique solution $(\rho,\eta)$ to \eqref{eq:ivp} such that $(\rho_0,\eta_0)=(\rho^0,\eta^0)$.
\end{theorem}
\begin{proof}
We refer the reader to~\cite[Lemma 3.3]{esposito_mikolas_2024} and \cite[Theorem 3.6]{esposito_mikolas_2024}.     
\end{proof}
As shown in \cite[Proposition 2.8]{esposito_mikolas_2024}~\eqref{eq:ivp} preserves mass, i.e. $\rho_t(\Rd) = \rho_0(\Rd)$ for $t\in[0,T]$. In the remainder of this article, for any curve $\gamma \in C([0,T],S)$, for some normed space $(S, \norm{\cdot}_S)$, we will write 
\[
\norm{\gamma}_{\infty,S}:= \sup_{t \in [0,T]}\norm{\gamma_t}_{S} \ . 
\]
We denote by $\|\cdot\|_\infty$ the usual sup-norm when this does not create confusion. For the reader's convenience we postpone to the corresponding section the preliminary results we shall use for the long-time behaviour in section~\ref{subsec:long_time_behaviour}. 
\section{The Euler equation for~\eqref{eq:ivp}}\label{sec:euler-co-ncl}

We focus on the case when $\mu$ defines the set of all possible vertices so that it makes sense to consider $\rho\ll\mu$. In particular, for any $t\in[0,T]$, $\rho_t: = r_t\dd\mu$, for a curve $r \in C([0,T],L^2_\mu(\Rd))$, and $\mu\in\P_2(\Rd)$. In this setting, the nonlocal continuity equation in~\eqref{eq:ivp} can be rewritten so that it can be interpreted as a nonlinear Euler equation. We observe that in this set up there is an interesting comparison between the study carried out in~\cite{Esposito_Pattachini_Schlichting_Slepcev,Esposito_on_a_class,esposito_mikolas_2024} and the recent developments on heterogeneous interacting particle systems, e.g., in~\cite{paul2024microscopic,ayi2024large}. Indeed, first we note that, by choosing $\lambda = \mu \otimes \mu$ in Definition \ref{def:flux}, since $\rho\ll\mu$, we can write the flux as
\[
\dd F^{\Phi}[\mu, \eta; g , w](x, y)=\Phi(g(x), g(y) ; w(x, y))\eta(x,y) \mathrm{d}(\mu \otimes \mu)(x, y),
\]
for any $g \in L^2_\mu(\mathbb{R}^d),\ w \in \mathcal{V}^{\text{as}}(\Rddiag)$, and $(x, y) \in \Rddiag$,  so that $\Phi$ is function of the density, $g$, w.r.t. $\mu$. Furthermore, if we assume that $\Phi$ is jointly antisymmetric, that is, $\Phi(a, b ;-v)=-\Phi(b, a ; v)$ for any $a, b, v \in \mathbb{R}$, and $\eta$ symmetric, the nonlocal divergence of $F^{\Phi}[\mu, \eta ; g, V[g]]$ is given by
\begin{equation}\label{eq:small_flux}
\onabla \cdot F[\mu, \eta ; g, V[g]](x)=\int_{\mathbb{R}^d \backslash \{x\}} \Phi(g(x), g(y) ; V[g](x, y)) \eta(x, y) \mathrm{d} \mu(y)   \ , 
\end{equation}
for $\mu \text {-a.e. } x \in \mathbb{R}^d$. In view of this consideration, we shall consider the system
\begin{equation}\label{eq:euler}
   \begin{split} 
    \partial_t r_t(x) & = - \int_{\Rd\backslash\{x\}}\Phi(r_t(x),r_t(y); V_t[r](x,y))\eta_t(x,y)\dd\mu(y)
    \\
    \partial_t \eta_t & =  \w_t[r] -\eta_t, 
    \end{split}\tag{Euler Co-NCL}
\end{equation}
for given initial data $r^0 \in L^2_\mu(\R^d)$ and $\eta^0\in C_{b}(\Rddiag)$ and symmetric. Solutions of \eqref{eq:euler} are defined as follows, {either in $L^2_\mu(\Rd)$ or $L^\infty_\mu(\Rd)$ since we shall need both setting.}
\begin{definition}[Solution to~\eqref{eq:euler}]\label{def:sol_to_euler}
   {Let $p=2,\infty$.} Given a jointly antisymmetric admissible flux interpolation $\Phi$, a velocity field $V:[0,T]\times L^p_\mu(\Rd) \to \mathcal{V}^{as}(\Rddiag)$, and a function $\w:[0,T]\times L^p_\mu(\R^d)\times \Rddiag \to \R$, a pair $(r, \eta): [0,T] \to L^p_\mu(\R^d)\times  C_{b}(\Rddiag)$ is a solution~\eqref{eq:euler} with initial datum $(r_0,\eta_0)\in L^p_\mu(\Rd)\times C_b(\Rd)$ if it satisfies the following properties: 
    \begin{enumerate}
   \item the maps $t \mapsto \int_{\Rd\backslash\{x\}}\Phi(r_t(x),r_t(y); V_t[r](x,y))\eta_t(x,y)\dd\mu(y)$ and $t \mapsto  \w_t[\rho_t] -\eta_t$ belong to $L^1([0,T])$;  \label{cond:sol_1_r}
       \item for $\mu$-a.e. $x \in \Rd$, $t \in [0,T]$, $(x,y)\in\Rddiag$, it holds 
    \begin{subequations}\label{eq:sol_r,eta}
       \begin{align}
             r_t(x) & = r_0(x) - \int_{0}^t \int_{\Rd\backslash\{x\}}\Phi(r_s(x),r_s(y); V_s[r](x,y))\eta_s(x,y)\dd\mu(y)\dd s
           \\
           \label{eq:sol_eta}
           \eta_t(x,y) &= \eta_0(x,y)  + \int_0^t  \left(\w_s[r](x,y) -\eta_s(x,y)\right) \,\dd s.
        \end{align}
    \end{subequations}
    \item $r \in AC([0,T], L^p_\mu(\R^d)),\ \eta \in AC([0,T],  C_{b}(\Rddiag))$. \label{cond:sol_ac_r,eta}
\end{enumerate}
\end{definition}

Possible examples, in the $L^2_\mu$ setting, for the functions $V$ and $\w$ are given by
 \begin{align}\label{eq:omega_example}
        \w_t[r](x,y) &= \int_{\Rd}W(t,x,y,z) r(z)\dd\mu(z),
        \\
        \label{eq:V_example}
    V_t[r](x,y) &= - \onabla (K*\rho_t)(x,y)= -\int_{\Rd} (K(y,z)- K(x,z))r_t(z) \dd \mu(z),
    \end{align}
 where for $(x,y) \in \Rddiag$, the map $t\mapsto W(t,\cdot,\cdot,\cdot) \in L^1([0,T])$, and $W \in C_b([0,T]\times \Rddiag\times \Rd)$, $K \in C_b(\R^2)$ and the velocity defined in~\eqref{eq:V_example} is related to the nonlocal interaction energy $\mathcal{E}(\rho)=(1/2)\iint K(x,y)\dd\rho(x)\dd\rho(y) $.

{For $p=2,\infty$}, we obtain well-posedness of~\eqref{eq:euler} by exploiting the Banach fixed-point theorem in the space $C([0,T], {L^p_\mu(\Rd)})\times C([0,T], C_b(\Rddiag))$ equipped with the distance 
\[
\tilde{d}_\infty((f^1,g^1), (f^2,g^2)) := \norm{f^1-f^2}_{\infty, {L^p_\mu(\Rd)}} + \norm{g^1-g^2}_{\infty, C_b(\Rddiag)},
\]
\begin{align*}
    \norm{f^1-f^2}_{\infty, {L^p_\mu(\Rd)}} &:= \sup_{t \in [0,T]} \norm{f^1_t-f^2_t}_{{L^p_\mu(\Rd)}} \mbox{ and }\norm{g^1-g^2}_{\infty, C_b(\Rddiag)} & := \sup_{t \in [0,T]} \norm{g^1_t-g^2_t}_{C_b(\Rddiag)} \ . 
\end{align*}
In this section, we fix $p=2$ and proving \textit{a priori} properties for~\eqref{eq:euler} useful to define the solution maps.  
\begin{proposition}\label{prop:L2_apriori_props}
   Let $\Phi$ be a jointly antisymmetric admissible flux interpolation, $r_0 \in L^2_{\mu}(\Rd)$, $\eta_0\in C_b(\Rddiag)$ and symmetric. Assume $\w:[0,T]\times L^2_\mu(\R^d)\times \Rddiag  \to  \R$ is such that the map $(x,y)\in\Rddiag  \mapsto \w_t[\cdot](\cdot,x,y)$ is continuous and that for a constant $C_\w>0$
\[
       \int_0^T \sup_{r \in L^2_\mu(\R^d)}\sup_{(x,y) \in \Rddiag}|\w_s[r](x,y)| \dd s\le C_\w\ .
\]
Suppose that $V:[0, T]\times L^2_\mu(\R^d) \rightarrow \mathcal{V}^{\mathrm{as}}(\Rddiag)$ satisfies 
    \begin{equation}\label{eq:V_bound}
        \int_0^T \sup_{r \in L^2_\mu(\Rd)}\sup_{x \in \Rd} \int_{\Rd\backslash \{x\}}|V_t[r](x,y)|^2\dd\mu(y)\dd t \leq C_V, 
    \end{equation}
 for some $C_V>0$. For a pair $(r,\eta):[0,T]\to L^2_\mu(\Rd)\times C_b(\Rddiag)$ satisfying~\eqref{eq:sol_r,eta} the following properties hold.
 \begin{enumerate}
     \item For $\mu$-a.e. $x \in \Rd$ and every $(x,y)\in\Rddiag$, the maps 
     \begin{align*}
         &t \mapsto\int_{\Rdminus}\Phi(r_t(x),r_t(y); V_t[r](x,y))\eta_t(x,x)\dd\mu(y) \in L^1([0,T]),
         \\
         & t\mapsto \w_t[r](x,y) - \eta_t(x,y) \in L^1([0,T]) ,
     \end{align*}
     
and $r \in L^\infty([0,T];L^2_\mu(\Rd)), \eta\in L^\infty([0,T],C_b(\Rddiag))$.
\item $r \in AC([0,T];L^2_\mu(\Rd))$ and $\eta \in AC([0,T];C_b(\Rddiag))$.
 \end{enumerate}
\end{proposition}
\begin{proof}
We obtain $\eta \in L^\infty([0,T]; C_b(\Rddiag))$ in an analogous way to~\cite[Proposition 2.8]{Esposito_on_a_class} using~\eqref{eq:eta_evolution}, with
\begin{equation}\label{eq:bound_eta_gronwall}
         \|\eta_t\|_{\infty}\leq \left(\|\eta_0\|_{\infty} + C_\w\right) e^{T}\ , 
    \end{equation}
    and note that the integrability of $t\mapsto \w_t[r](x,y) - \eta_t(x,y)$ directly follows. As for $r$, for $\mu$-a.e. $x\in\Rd$, by using~\ref{ass:interp_deg},~\ref{ass:interp_lip} in Definition~\ref{def:admissible_interpolation}, boundedness of $\eta$, and~\eqref{eq:V_bound} we have 
    \begin{align*}
        |r_t(x)| &\leq |r_0(x)| + \int_{0}^t\int_{\Rdminus}\left|\Phi(r_s(x),r_s(y); V_s[r](x,y))\eta_s(x,y)\right|\dd\mu(y)\dd s
        \\
        & \leq |r_0(x)| +L_\Phi\norm{\eta}_{\infty, C_b(\Rddiag)} \int_{0}^t\int_{\Rdminus}\left|V_s[r](x,y)\right|(|r_s(x)| + |r_s(y)|)\dd\mu(y)\dd s
        \\
        & \leq |r_0(x)|\! +\!L_\Phi\norm{\eta}_{\infty, C_b(\Rddiag)}\! \int_{0}^t\!\left(\sup_{r\in L^2_\mu(\Rd)}\sup_{x \in \Rd}\int_{\Rdminus}\left|V_s[r](x,y)\right|^2\dd\mu(y)\!\right)^{\frac{1}{2}}\!\!
        \\
        &\qquad  \times \left(\!\int_{\Rdminus}\!\!\!(|r_s(x)|+ |r_s(y)|)^2\dd\mu(y)\right)^{\frac{1}{2}}\!\!\!\! \dd s
        \\
        & \leq |r_0(x)| + L_\Phi\norm{\eta}_{\infty, C_b(\Rddiag)} \int_{0}^t\bar{v}_s^{1/2}({\norm{r_s}_{L^2_\mu(\Rd)}} + |r_s(x)|) \dd s\ ,
    \end{align*}
where we used the Cauchy--Schwarz and Minkowski inequalities, $\mu(\Rd)=1$, and we set 
\begin{equation}\label{eq:small_v_notation}
\bar{v}_t:=\sup_{r\in L^2_\mu(\Rd)}\sup_{x \in \Rd}\left(\int_{\Rdminus}|V_t[r](x,y)|^2 \dd\mu(y)\right)\ . 
\end{equation}
An application of Minkowski's inequality and \cite[Theorem 6.19 a.]{folland1999real} yields
\[
\norm{r_t}_{L^2_\mu(\Rd)} \leq \norm{r_0}_{L^2_\mu(\Rd)} + L_\Phi\norm{\eta}_{\infty, C_b(\Rddiag)} 2\int_{0}^t\bar{v}_s^{1/2}{\norm{r_s}_{L^2_\mu(\Rd)}} \dd s \ .
\]
Then, an application of Gr\"onwall's inequality yield, for any $t\in[0,T]$, 
\begin{equation}\label{eq:L2_r_bounded}
\norm{r_t}_{L^2_\mu(\Rd)} \leq \norm{r_0}_{L^2_\mu(\Rd)}e^{2L_\Phi\norm{\eta}_\infty C_V},
\end{equation}
thus $r \in L^{\infty}([0,T];L^2_\mu(\Rd))$. The integrability of the nonlocal divergence can be proven by a further application of Gronwall inequality. More precisely, by the previous argument we know that, for $\mu$-a.e. $x\in\Rd$,
\begin{align*}
|r_t(x)|&\le|r_0(x)|+\int_{0}^t\int_{\Rdminus}\left|\Phi(r_t(x),r_t(y); V_s[r_s](x,y))\eta_s(x,y)\right|\dd\mu(y)\dd s
\\
& \leq  {|r_0(x)| + L_\Phi\norm{\eta}_{\infty, C_b(\Rddiag)}\left(\norm{r_0}_{L^2_\mu(\Rd)}e^{2L_\Phi\norm{\eta}_\infty C_V}  \int_{0}^t\bar{v}_s^{1/2}\dd s+ \int_{0}^t\bar{v}_s^{1/2}|r_s(x)| \dd s\right)}
\\
& \leq  {|r_0(x)| + L_\Phi\norm{\eta}_{\infty, C_b(\Rddiag)}\left(  \norm{r_0}_{L^2_\mu(\Rd)}e^{2L_\Phi\norm{\eta}_\infty C_V}\left(t\int_{0}^t\bar{v}_s\dd s \right)^{1/2}+ \int_{0}^t\bar{v}_s^{1/2}|r_s(x)| \dd s\right)}
\\
&\le|r_0(x)|+\bar{C}\sqrt{t}+L_\Phi\norm{\eta}_{\infty, C_b(\Rddiag)}\int_{0}^t\bar{v}_s^{1/2}|r_s(x)| \dd s, 
\end{align*}
for $\bar{C}=L_\Phi\norm{\eta}_{\infty, C_b(\Rddiag)}\norm{r_0}_{L^2_\mu(\Rd)}e^{2L_\Phi\norm{\eta}_\infty C_V}C_V^{1/2}$. Then, for any $t\in[0,T]$ and $\mu$-a.e. $x\in\Rd$, 
\[
|r_t(x)|\le (|r_0(x)|+\bar{C}\sqrt{t})\exp\left(L_\Phi\norm{\eta}_{\infty, C_b(\Rddiag)}\int_0^t\bar{v_s}^{1/2} \dd s\right) \le (|r_0(x)|+\bar{C}\sqrt{t})e^{L_\Phi\norm{\eta}_{\infty, C_b(\Rddiag)}\sqrt{t}\sqrt{C_V}},
\]
from which we infer time-integrability of the nonlocal divergence, since we can bound ${\norm{r_s}_{L^2_\mu(\Rd)}} + |r_s(x)|$ on $[0,T]$. Part (2) is consequence of the time integrability and~\eqref{eq:sol_r,eta}. 
\end{proof}
We proceed with proving well-posedness of \eqref{eq:euler}. Fix $r_0\in L^2_\mu(\Rd),\  \eta_0 \in C_b(\Rddiag)$ for $\eta_0$ symmetric, and assume $\Phi$ is a jointly antisymmetric admissible flux interpolation as in Definition~\ref{def:admissible_interpolation}. We consider 
\[
\ACt:=AC([0,T]; L^2_{\mu}(\R^d)) \times AC([0,T]; C_b(\Rddiag)),
\]
and equip it with the distance $\tilde{d}_\infty$. We define the solution map $\S:= (\S_V,\S_\w): \ACt \to \ACt$ as
\begin{align*}
    \S_V(r,\eta)(t)(x)&:= r_0(x) - \int_0^t \int_{\Rdminus} \Phi(r_s(x), r_s(y);V_s[r](x,y))\eta_s(x,y)\dd \mu(y)\dd s, \quad \mu\text{-a.e.}\ x \in \Rd,
    \\
    \S_\w(r,\eta)(t) (x,y)& := \eta_0(x,y) + \int_0^t \left(\w_s[r](x,y) - \eta_s(x,y)\right) \dd s\ , 
\end{align*}
for any $t \in [0,T]$ and $(x,y) \in \Rddiag$. The a priori properties in Proposition~\ref{prop:L2_apriori_props} ensure this map $\S$ is well defined. Below, we prove existence and uniqueness of solutions to~\eqref{eq:euler} by assuming a uniform-in-time integrability assumption for the velocity field $V$ motivated by our guiding example, see \eqref{eq:V_example}, although the result holds for $V$ satisfying~\eqref{eq:V_bound}. The stronger assumption is relevant in Section~\ref{sec:graph_CE}, where it is used to obtain the Lipschitzness of the characteristics.
\begin{theorem}\label{thm:well_posedness_euler}
    Let $\Phi$ be a jointly antisymmetric admissible flux interpolation as in Definition~\ref{def:admissible_interpolation} and consider an initial datum $(r^0, \eta^0) \in L^2_\mu(\Rd)\times C_b(\Rddiag)$, with $\eta^0$ symmetric. Assume that $V:[0,T]\times L^2_\mu(\Rd)\to \V^{as}(\Rddiag)$ satisfies~\eqref{eq:V_bound} or  

\begin{subequations}
    \begin{equation}\label{eq:velocity_second_moment}
    \sup_{t\in[0,T]} \sup_{r \in L^2_\mu(\Rd)} \sup_{x\in \Rd}  \int_{\Rdminus}|V_t[r](x,x')|^2\dd\mu(x') \leq C_{V} ,
    \end{equation}
and, for all $t\in[0,T]$ and any $g,\tilde{g}\in L^2_\mu(\Rd)$,
\begin{equation}\label{eq:velocity_Lipschitz_second_moment}
     \sup_{x\in \Rd}  \int_{\Rd\backslash \{x\}}|V_t[g](x,x') - V_t[\tilde{g}](x,x')|^2\dd\mu(x') \leq L_V\norm{g-\tilde{g}}^2_{L^2_\mu(\Rd)},     
    \end{equation}
for constants $C_V, L_V>0$.
\end{subequations}    
Assume that $\w:[0,T]\times L^2_\mu(\Rd) \times \Rddiag \to \R$ satisfies \ref{ass:w_continuous}-\ref{ass:omega_symmetric}. Then, there exists a solution $(r, \eta)$ to \eqref{eq:euler} with initial datum $(r^0,\eta^0)$, as in~Definition \ref{def:sol_to_euler} . 
\end{theorem}
\begin{proof}
    Let us assume~\eqref{eq:velocity_second_moment} since the argument is slightly more involved for this case at the end of the proof. For $(r^0,\eta^0)$ given, we consider $(r,\eta), (\tr,\teta) \in \ACt$ and, for $\mu$-a.e. $x \in \Rd$, we have
    \begin{align*}
        &|\S_V(r,\eta)(t) - \S_V(\tr,\teta)(t)|(x)  \\
        &\leq \int_0^t \int_{\Rdminus}\Big| \Phi(r_s(x), r_s(y);V_s[r](x,y))\eta_s(x,y) - \Phi(\tr_s(x), \tr_s(y);V_s[\tr](x,y))\teta_s(x,y)\Big|\dd \mu(y)\dd s
        \\
        & \leq L_\Phi \int_0^t  \int_{\Rdminus}\big|\eta_s(x,y)\big| (|r_s(x) - \tr_s(x)| + |r_s(y) - \tr_s(y)|)|V_s[r](x,y)|\dd \mu(y)\dd s
        \\
        &\quad+L_\Phi \int_0^t  \int_{\Rdminus}\left(|\tr_s(x)| + |\tr_s(y)|\right)|V_s[r](x,y)|\big|\eta_s(x,y) - \teta_s(x,y)\big| \dd \mu(y)\dd s 
        \\
        &\quad + L_\Phi \int_0^t  \int_{\Rdminus}\big|\teta_s(x,y)\big|\left(|\tr_s(x)| + |\tr_s(y)|\right)|V_s[r](x,y) - V_s[\tr](x,y)|\dd \mu(y)\dd s
        \\
        &=: I + II + III,
    \end{align*}
where we used~\ref{ass:interp_deg},~\ref{ass:interp_lip} in Definition~\ref{def:admissible_interpolation}.
$I$ can be estimated as follows by means of Cauchy--Schwarz and Minkowski inequalities
\begin{align}\label{eq:term_1_euler}
    |I| &\leq L_\Phi \norm{\eta}_{\infty, C_b(\Rddiag)} \left(\sup_{t \in [0,T]}\sup_{r \in L^2_\mu(\Rd)}\sup_{x\in\Rd}\int_{\Rdminus}\left|V_t[r](x,y)\right|^2\dd\mu(y)\right)^{1/2} \nonumber
    \\ &\qquad\qquad\qquad\times\int_{0}^t\left(\int_{\Rdminus}\left(|r_s(x)-\tr_s(x)| + |r_s(y) - \tr_s(y)|\right)^2\dd\mu(y)\right)^{1/2}\!\!\! \dd s
        \\
       & \leq  L_\Phi\norm{\eta}_{\infty, C_b(\Rddiag)}C_V^{1/2}\int_{0}^t(\norm{r_s-\tr_s}_{L^2_\mu(\Rd)} + |r_s(x) - \tr_s(x)|)\dd s\ , \nonumber 
\end{align}
where in the second inequality we used \eqref{eq:velocity_second_moment}. Squaring both sides and integrating with respect to $\mu$
\begin{align*}
    \norm{I}^2_{L^2_\mu(\Rd)} & \leq L_\Phi^2 \norm{\eta}_{\infty,C_b(\Rddiag)}^2 C_V \int_{\Rd}\bigg|\int_{0}^t(\norm{r_s-\tr_s}_{L^2_\mu(\Rd)} + |r_s(x) - \tr_s(x)|)\dd s\bigg|^2 \dd\mu(x)
    \\
    & \leq  L_\Phi^2 \norm{\eta}_{\infty,C_b(\Rddiag)}^2C_V T\int_{\Rd}\int_0^T(\norm{r_s-\tr_s}_{L^2_\mu(\Rd)} + |r_s(x) - \tr_s(x)|)^2\dd s\dd\mu(x)
    \\
    & \leq L_\Phi^2 \norm{\eta}_{\infty,C_b(\Rddiag)}^2 2C_V T\int_0^T \norm{r_s-\tr_s}^2_{L^2_\mu(\Rd)} \dd s\ 
    \\
    & \leq L_\Phi^2 \norm{\eta}_{\infty,C_b(\Rddiag)}^2 2C_V  T^2\norm{r-\tr}^2_{\infty, L^2_\mu(\Rd)},
\end{align*}
thus
\[
\norm{I}_{L^2_\mu(\Rd)} \leq \sqrt{2} L_\Phi \norm{\eta}_{\infty, C_b(\Rddiag)} C_V^{1/2}\norm{r - \tr}_{\infty, L^2_\mu(\Rd)}{T}\ . 
\]
For $II$ we obtain similarly that 
\[
|II| \leq L_\Phi \norm{\eta -\teta}_{\infty, C_b(\Rddiag)}C_V^{1/2}\int_0^t(  \norm{\tr_s}_{L^2_\mu(\Rd)}+ |r_s(x)|)\dd s\ ,
\]
and thus 
\begin{align*}
\norm{II}_{L^2_\mu(\Rd)} &\leq L_\Phi \norm{\eta -\teta}_{\infty, C_b(\Rddiag)}C_V^{1/2}(\norm{r}_{\infty,L^2_\mu(\Rd)} + \norm{\tr}_{\infty,L^2_\mu(\Rd)}) T\ 
\\
& \leq 2 L_\Phi \norm{\eta -\teta}_{\infty, C_b(\Rddiag)}C_V^{1/2}\bar{C}_T T  \ ,
\end{align*}
where in the last inequality $\bar{C}_T=\norm{r_0}_{L^2_\mu(\Rd)}e^{2L_\Phi\norm{\eta}_\infty C_V T}$, from~\eqref{eq:L2_r_bounded}. Finally, for $III$ we use~\eqref{eq:velocity_Lipschitz_second_moment} to have 
\begin{align*}
    |III| &\leq L_\Phi \norm{\teta}_{\infty, C_b(\Rddiag)}\left(\sup_{t\in[0,T]}\sup_{r,\tr \in L^2_\mu(\Rd)}\sup_{x \in \Rd} \int_{\Rdminus}|V_t[r](x,y) - V_t[\tr](x,y)|^2 \dd \mu(y)\right)^{1/2} 
    \\
    &\qquad \qquad \qquad \qquad \qquad \times \int_0^t \norm{\tr_s}_{L^2_\mu(\Rd)}  + |\tr_s(x)|\dd s\ 
    \\
    & \leq L_\Phi\norm{\teta}_{\infty, C_b(\Rddiag)} L_V^{1/2} \norm{r - \tr}_{\infty, L^2_\mu(\Rd)}\int_0^t \norm{\tr_s}_{L^2_\mu(\Rd)}  + |\tr_s(x)|\dd s. 
\end{align*}
Therefore, 
\begin{align*}
    \norm{III}_{L^2_\mu(\Rd)} &\leq \sqrt{2} L_\Phi\norm{\teta}_{\infty, C_b(\Rddiag)} L_V^{1/2}\norm{r}_{ \infty, L^2_\mu(\Rd)}\norm{r - \tr}_{\infty, L^2_\mu(\Rd)}T\ 
    \\
    & \leq \sqrt{2} L_\Phi\norm{\teta}_{\infty, C_b(\Rddiag)} L_V^{1/2}\norm{r - \tr}_{\infty, L^2_\mu(\Rd)}\bar{C}_TT.
\end{align*}
The same estimate for $\S_\w$ gives
\begin{align*}
    \norm{\S_\w(r,\eta) - \S_\w(\tr,\teta)}_{\infty, C_b(\Rddiag)} &\leq \norm{\eta -\teta}_{\infty, C_b(\Rddiag)} T + L_\w \norm{r - \tr}_{\infty, L^2_\mu(\Rd)}T,
\end{align*}
due to~\ref{ass:omega_lip}~and Proposition~\ref{prop:L2_apriori_props}, see also \cite[Lemma 3.3]{esposito_mikolas_2024}. 
A combination of the previous estimates implies that
\begin{align*}
        d_\infty((r,\eta),(\tr,\teta)) &\leq (\sqrt{2} L_\Phi\norm{\teta}_{\infty, C_b(\Rddiag)}( L_V^{1/2}\bar{C}_T+C_V^{1/2}) + L_\w))\norm{r - \tr}_{\infty, L^2_\mu(\Rd)} T 
    \\
    &\quad  + ( 1 + 2L_\Phi 
C_V^{1/2}\bar{C}_T)\norm{\eta - \teta}_{\infty, C_b(\Rddiag)}T\\
&\leq \underbrace{(\sqrt{2} L_\Phi(\norm{\eta_0}_{\infty}+C_\w T)e^T( L_V^{1/2}\bar{C}_T+C_V^{1/2}) + L_\w))}_{:=\alpha(T)}\norm{r - \tr}_{\infty, L^2_\mu(\Rd)} T
    \\
    &\quad  + \underbrace{( 1 + 2L_\Phi 
C_V^{1/2}\bar{C}_T)}_{:=\beta(T)}\norm{\eta - \teta}_{\infty, C_b(\Rddiag)}T
\end{align*}
where we used~\eqref{eq:bound_eta_gronwall} with~\ref{ass:omega_bounded}.
In particular, it follows that 
\[
d_\infty(\S(r,\eta),\S(\tr,\teta)) \leq \max\{\alpha(T),\beta(T)\}Td_\infty((r,\eta),(\tr,\teta))
\]
so that for $T < \kappa(T):=1/\max\{\alpha(T), \beta(T)\}$ the solution map is a contraction. By the Banach fixed-point theorem we can conclude there exists a unique solution to \eqref{eq:euler} in $C_T=C([0,T]; L^2_\mu(\R^d)) \times C([0,T]; C_b(\Rddiag))$ for $T<T^*$, where $T^*$ is such that $\kappa(T^*)^{-1}T^*=1$. This can be verified since $\alpha(T)\simeq 1+ e^T + Te^T$ and $\beta(T) \simeq  1 + e^{T^2e^T}$. Absolute continuity in time follows since $\onabla \cdot F[\mu, \eta ; r, V[r]]$ and the map $(t\mapsto \w_t[r_t]-\eta_t)$ belong to $L^1([0,T])$, as proven in~Proposition~\ref{prop:L2_apriori_props}. The solution can be extended by using a standard iteration procedure so that the time interval does not depend on any $T^*$, see e.g.,~\cite[Theorem 3.6]{esposito_mikolas_2024}.

Now we turn to the analysis of the case $V$ satisfying~\eqref{eq:V_bound}. Following the same strategy, upon using~\eqref{eq:V_bound} instead of~\eqref{eq:velocity_second_moment}, we obtain the bound 
\begin{align*}
     d_\infty((r,\eta),(\tr,\teta))  &\leq \underbrace{(\sqrt{2} L_\Phi(\norm{\eta_0}_{\infty}+C_\w T)e^T( L_V^{1/2}\bar{C}+C_V^{1/2}) + L_\w))}_{:=\alpha(T)}\norm{r - \tr}_{\infty, L^2_\mu(\Rd)} (T+\sqrt{T})
\\
     &\quad  + \underbrace{( 1 + 2L_\Phi 
 C_V^{1/2}\bar{C})}_{:=\beta}\norm{\eta - \teta}_{\infty, C_b(\Rddiag)}(T+\sqrt{T})\ ,
\end{align*}
where $\bar{C}=\norm{r_0}_{L^2_\mu(\Rd)}e^{2L_\Phi\norm{\eta}_\infty C_V}$ and thus 
\[
d_\infty(\S(r,\eta),\S(\tr,\teta)) \leq \max\{\alpha(T),\beta\}(T+\sqrt{T})d_\infty((r,\eta),(\tr,\teta))\ ,
\]
 so that for $T+\sqrt{T} < \kappa(T):=1/\max\{\alpha(T), \beta\}$ the solution map is a contraction. The rest of the argument follows the previous case in a completely analogous way.
\end{proof}
\begin{remark}\label{rem:antisymmetry_eta_interpolation}
   The assumptions on the joint antisymmetry of $\Phi$ and the symmetry of $\eta$ through \ref{ass:omega_symmetric} and $\eta_0$ symmetric in Theorem \ref{thm:well_posedness_euler} are needed so that \eqref{eq:ivp} can be expressed as \eqref{eq:euler}, see~\cite[Section 5]{Esposito_on_a_class}. However, these assumptions are not required to obtain well-posedness of the dynamics described by  \eqref{eq:euler} itself.
\end{remark}
Since we proved existence and uniqueness of solutions to \eqref{eq:euler}, we note that equation~\eqref{eq:sol_eta} admits the explicit solution depending on $r$:
 \begin{equation}\label{eq:explicit_solution_eta}
    \eta_t[r](x, y)=e^{-t }\left(\eta_0(x, y)+\int_0^t e^{s} \omega_s\left[r\right](x, y) \mathrm{d} s\right)\ .
 \end{equation}
Thus, the bound on $\norm{\eta}_{\infty,C_b(\Rddiag)}$ can be improved to 
 \begin{equation}\label{eq:norm_eta_bound}
     \norm{\eta}_{\infty,C_b(\Rddiag)} \leq \norm{\eta_0}_{C_b(\Rddiag)} + C_\w \ .
 \end{equation}
Having proved the well-posedness of \eqref{eq:euler} we turn to the study of the graph continuity equation in the next section.

\section{The graph continuity equation}\label{sec:graph_CE}
Throughout this section we consider a co-evolving graph, $\mathcal{G}_t=(\mu,\eta_t[r])$, for $(r,\eta)$ a solution to~\eqref{eq:euler} in the sense of Definition \ref{def:sol_to_euler}, which exists and is unique due to Theorem~\ref{thm:well_posedness_euler}. We observe that~\eqref{eq:euler} is linked to a particular nonlinear continuity equation if we consider the evolution of the measure $\delta_{r_t}\otimes\mu$, see also~\cite[Section 3.1.3]{paul2024microscopic} for a similar case. More precisely, for any Borel set $A\times B \in \B(\R\times\Rd)$ let us consider the measure $\sigma:= \delta_r \otimes \mu$ given by 
\begin{equation}\label{eq:mu_monokinetic}
   ( \delta_r\otimes \mu)(A \times B) = \int_{B}\delta_{r(x)}(A) \dd \mu(x), 
\end{equation}
where $(r, \eta)$ is the solution of \eqref{eq:euler} with base measure $\mu \in \P_2(\Rd)$. One can check that $\delta_{r_t}\otimes\mu$ is a weak solution (in the sense of Definition~\ref{def:sol_vlasov}) of the following initial value problem  
 \begin{equation}\label{eq:vlasov_equation}
\begin{cases}
  \partial_t \sigma + \partial_{\xi}(\sigma\X[\sigma, r, \eta]) = 0,
    \\
     \sigma_0 = \bar{\sigma} \in \P(\R \times \Rd),
\end{cases}\tag{graph-CE}
\end{equation}
where, for any $\sigma \in \P(\R\times\Rd)$ the velocity field is
\begin{equation}\label{eq:mean_field}
\X[\sigma,r, \eta](t,\xi,x):=-\int_{\R \times \Rdminus}\Phi( \xi, \xi'; V_t[r](x,x'))\eta_t(x,x')\dd \sigma(\xi',x')  \ , 
\end{equation}
for a solution $(r,\eta)$ to~$\eqref{eq:euler}$. We refer to~\eqref{eq:vlasov_equation} as the \textit{graph-continuity equation} since the velocity field depends on the solution of the co-evolving graph problem. We devote the rest of this section to the study of its well-posedness. Solutions to~\eqref{eq:vlasov_equation} are understood in the distributional sense. 
\begin{definition}\label{def:sol_vlasov}
{Let $\mathcal{G}_t(\mu,\eta_t[r])$ be a given co-evolving graph, where $(r,\eta)$ solve \eqref{eq:euler}}. A curve $\sigma \in C([0,T], \P(\R\times \Rd))$ is a solution to \eqref{eq:vlasov_equation} with initial data $\sigma_0 \in \P(\R\times \Rd)$ if, for any $\varphi \in C_c^1([0,T] \times\R \times \Rd)$, it holds
\begin{equation}\label{eq:sol_concept_vlasov}
\begin{aligned}
&\int_0^T\int_{\R \times \Rd} \partial_t\varphi(t, \xi, x) d \sigma_t(\xi, x)\dd t + \int_{\R\times\Rd} \varphi(0,\xi,x) \dd \sigma_0(\xi,x) - \int_{\R \times \Rd} \varphi(T,\xi,x) \dd \sigma_{T}(\xi,x)
\\
& =\int_0^T \int_{\R \times \Rd}  \int_{\R \times \Rdminus}\partial_{\xi} \varphi(t,\xi, x)\Phi( \xi, \xi'; V_t[r](x,x'))\eta_t(x,x') \dd \sigma_t(\xi^{\prime}, x^{\prime}) \dd \sigma_t(\xi, x) \dd t
\ . 
\end{aligned}
\end{equation}
\end{definition}
We are interested on evolutions on graphs, whose possible vertices are given by $\mu$, which is constant in time. For this reason, it makes sense to consider the disintegration of a solution $\sigma_{t}$ to~\eqref{eq:vlasov_equation} of the form \eqref{eq:mu_monokinetic} with respect to the marginal $\mu$ on $\Rd$, i.e. $\sigma_{t} = \int_{\Rd} \sigma_{t,x}\dd \mu(x)$. In what follows, we will denote by $\pi^1:\R\times \Rd\to\R $ and $\pi^2:\R\times\Rd \to \Rd$ the canonical projections given by $\pi^1(x,x') = x$ and $\pi^2(x,x') = x'$ for any $(x,x') \in \R\times\Rd$, so that $\mu = \pi^2_\# \sigma$ for any measure $\sigma$  on $\R\times \Rd$ with marginal $\mu$ on $\Rd$.   By uniqueness $\mu$-almost everywhere of the disintegration we consider the \textit{disintegrated version} of~\eqref{eq:vlasov_equation}
\begin{equation}\label{eq:disintegrated_vlasov_equation}
\begin{cases}
\partial_t \sigma_{t,x} + \partial_\xi(\sigma_{t,x}\mathcal{X}[\sigma,r,\eta](t,\cdot, x)) = 0,\\
\sigma_{0,x}=\bar{\sigma}_x,
\end{cases}
\end{equation}
for $\mu$-almost every $x\in\Rd$ and $\bar\sigma=\int_{\Rd}\bar{\sigma}_x\dd\mu(x)$. For $\mu$-a.e. $x\in\Rd$ we interpret the above equation in the distributional sense, i.e., for any $\varphi\in C^1_c([0,T]\times\R\times\Rd)$, 
\begin{equation}\label{eq:mu_ae_sol_concept_vlasov}
\begin{aligned}
\int_0^T\int_{\R } &\partial_t\varphi(t, \xi, x) d \sigma_{t,x}(\xi)\dd t + \int_{\R} \varphi(0,\xi,x) \dd \sigma_{0,x}(\xi) - \int_{\R} \varphi(T,\xi,x) \dd \sigma_{T,x}(\xi)
\\
& =\int_0^T \int_{\R}  \int_{\R \times \Rdminus}\partial_{\xi} \varphi(t,\xi, x)\Phi( \xi, \xi'; V_t[r](x,x'))\eta_t(x,x') \dd \sigma_t(\xi^{\prime}, x^{\prime}) \dd \sigma_{t,x}(\xi) \dd t . 
\end{aligned}
\end{equation}
Note that integrating \eqref{eq:mu_ae_sol_concept_vlasov} with respect to $\mu$  we obtain $\sigma_t=\int_{\Rd}\sigma_{t,x}\dd\mu(x)$ is a distributional solution to~\eqref{eq:vlasov_equation}.
Due to the nature of our problem, we introduce the notation
\[
\P^\mu_2(\R\times\Rd):=\{\nu \in \P_2(\R\times\Rd) \ | \ {\pi_2}_\# \nu = \mu \}\ ,
\]
for measures in $\P_2(\R\times\Rd)$ having the marginal $\mu$ on $\Rd$. This means we will consider the embedding of the dynamics \eqref{eq:euler} into \eqref{eq:vlasov_equation} for graphs with the same base measure $\mu \in \P_2(\Rd)$. We equip $\P_2^\mu(\R\times\Rd)$ with the distance
\[
L^2_\mu d_2(\nu^1,\nu^2) :=\left( \int_{\Rd}d^2_2(\nu^1_x,\nu^2_x) \dd\mu(x)\right)^{\frac{1}{2}} , \quad \nu^1,\nu^2 \in \P^\mu_2(\R \times \Rd)\ ,
\]
which forms a complete metric space, {see~\cite[Proposition A.0.1]{mikol2024a}}  for the formal statement and proof. For $\sigma \in \P_2^\mu(\R\times\Rd)$, we denote by $m_i[\sigma_{t,x}]:=\int_{\R}|\xi|^i\dd\sigma_{t,x}(\xi)$ $\mu$-a.e. $x \in \Rd$, and $t \in [0,T]$ the $i$-th moment of the disintegration of $\sigma$ with respect to $\mu$, for $i=1,2$, and we will denote by $m_i[\gamma_*]:= \sup_{t \in [0,T]}m_i[\gamma_t]$ the bound in time for the $i$-th moment of a curve $\gamma \in C([0,T];\P(\R))$. In order to prove existence and uniqueness of weak solutions we proceed by looking at the characteristics' equation, as usual for continuity-type equation.

\begin{lemma}\label{lemma:existence_of_flow}
    Let $V:[0,T]\times L^2_\mu(\Rd) \to \V^{as}(\Rddiag)$ be continuous in time and assume it satisfies \eqref{eq:velocity_second_moment} and \eqref{eq:velocity_Lipschitz_second_moment}. Let $\sigma \in C([0,T], \P^{\mu}_2(\R{\times\Rd}))$ with ${\pi_1}_\#\sigma=\gamma\in C([0,T];\P(\R))$ and $(r,\eta)$ be the solution of $\eqref{eq:euler}$. Then, for $\mu$-a.e. $x \in \Rd$, the ODE initial value problem
\begin{align*}
        \frac{d}{dt} v_t(x) &=  \X[\sigma,r,\eta](t,v_t(x), x) \ ,
        \\
        v_0(x) &= u \in \R \ , 
    \end{align*}
    has a unique solution on $[0,T]$.  

\end{lemma}
\begin{proof}
    We begin by showing that our assumptions imply the map  $\R \ni \xi \mapsto\X[\sigma,r,\eta](t,\xi,x)$ is globally Lipschitz, for any $(t,x)\in[0,T]\times\Rd$. Indeed, for any $a,b \in \R$ we use~\ref{ass:interp_deg},~\ref{ass:interp_lip} from Definition \ref{def:admissible_interpolation}, and~\eqref{eq:velocity_second_moment} to obtain
    \begin{equation}\label{eq:Lipschitz_velocity_field}
    \begin{split}
        |\X[\sigma,r,\eta]&(t,a,x) - \X[\sigma,r,\eta](t,b,x)|\\
        &\leq \  \int_{\R\times \Rdminus} \bigg|(\Phi(a,\xi';V_t[r](x,x'))- \Phi(b,\xi';V_t[r](x,x'))\eta_t(x,x')\bigg|\dd \sigma_t(\xi',x')
        \\
        & \leq  L_\Phi\int_{\R \times \Rdminus} |\eta_t(x,x')|\  |V_t[r](x,x')|\ |a-b| \dd \sigma_{t,x'}(\xi')\dd\mu(x')
        \\
       & \leq L_\Phi C_V^{1/2} \norm{\eta}_{\infty, C_b(\Rddiag)} |a - b|\ .
    \end{split}
    \end{equation}
    Let us check that $t \mapsto\X[\sigma,r,\eta](t,a,x)$ is continuous. Without loss of generality, for any $(a,x) \in \R \times \Rd$ and $s<t \in [0,T]$, consider 
    \begin{align*}
        &|\X[\sigma,r,\eta](t,a,x)  - \X[\sigma,r,\eta](s,a,x)| 
        \\
        &= \bigg|\int_{\R\times \Rdminus}\!\!\Phi(a,\xi';V_t[r](x,x')\eta_t(x,x')\dd \sigma_t(\xi',x')\! -\! \int_{\R\times\Rdminus}\!\!\Phi(a,\xi';V_s[r](x,x'))\eta_s(x,x') \dd \sigma_s(\xi',x') \nonumber\bigg|
        \\
        &\leq \bigg|  \int_{\R\times \Rdminus} \Phi(a,\xi';V_t[r](x,x')\eta_t(x,x')\dd (\sigma_t-\sigma_s)(\xi',x') \bigg|
        \\
        & \qquad + \int_{\R\times \Rdminus}\bigg| \Phi(a,\xi';V_t[r](x,x'))\eta_t(x,x') - \Phi(a,\xi';V_s[r](x,x'))\eta_s(x,x')\bigg| \dd \sigma_s(\xi',x') 
       \\
        & =: I + II\ .
        \end{align*}
    We begin by estimating $I$. Note that for $(\mu  \otimes \mu)$-a.e. $(x,x') \in \Rddiag$, $a\in\R$, and any $t \in [0,T]$ we have 
    \begin{equation}\label{eq:lip_constant_bound}
        \begin{split}
        Lip[\Phi&(a,\cdot,V_t[r](x,x'))\eta_t(x, x')]
        \\
        &:=\sup_{\xi^1\neq\xi^2\in \R}\frac{|(\Phi(a,\xi^1;V_t[r](x,x')-\Phi(a,\xi^2;V_t[r](x,x'))\eta_t|}{|\xi^1-\xi^2|} 
        \\
        & \leq L_\Phi|\eta_t(x,x')||V_t[r](x,x')|
        \\
        & \leq L_\Phi\norm{\eta}_{\infty, C_b(\Rddiag)}|V_t[r](x,x')|
        \ , 
        \end{split}
    \end{equation}
    similarly to \eqref{eq:Lipschitz_velocity_field}. Thus, the Lipschitz constant of the map $\xi \mapsto \Phi(a,\xi;V_t[r](x,x'))\eta_t(x,x')$ is bounded by $L_\Phi\norm{\eta}_{\infty, C_b(\Rddiag)}|V_t[r](x,x')|$ for $(\mu  \otimes \mu)$-a.e. $(x,x') \in \Rd\times\Rd$ and any $t \in [0,T]$. Then we can estimate $I$ by 
    \begin{align*}
       |I| & = \bigg|  \int_{ \Rdminus}\int_{\R} (\Phi(a,\xi';V_t[r](x,x'))\eta_t(x,x')\dd (\sigma_{t,x}-\sigma_{s,x})(\xi')\dd\mu(x') \bigg|
       \\
       &\leq  \int_{\Rdminus} \bigg|\int_{\R} \Phi(a,\xi';V_t[r](x,x'))\eta_t(x,x')\dd(\sigma_{t,x'}-\sigma_{s,x'})(\xi')\bigg| \dd\mu(x')
       \\
       & \leq  \int_{\Rd} Lip[\Phi(a,\cdot,V_t[r](x,x')\eta_t(x,x')]d_1(\sigma_{t,x'},\sigma_{s,x'})\dd\mu(x') 
        \\
       & \leq L_\Phi \norm{\eta}_{\infty,C_b(\Rddiag)} \int_{\Rd}|V_t[r](x,x')|d_1(\sigma_{t,x'},\sigma_{s,x'})\dd\mu(x') 
       \\
       & \leq L_\Phi \norm{\eta}_{\infty,C_b(\Rddiag)} C_V^{1/2}L^2_\mu d_2(\sigma_t,\sigma_s),
    \end{align*}
    where we used Cauchy-Schwarz inequality and \eqref{eq:velocity_second_moment}. By the continuity in time of $\sigma$ we have $I\to 0$ as $t\to s $. 
    Let us now focus on $II$. For any $a \in \R$ and $x \in \Rd$, note that $\Phi_t-\Phi_s:=\Phi(a,\xi';V_t[r](x,x'))\eta_t(x,x') - \Phi(a,\xi';V_s[r](x,x'))\eta_s(x,x')$ is integrable, since
        \begin{align*}
            |II| &\leq \int_{\R\times \Rdminus}\bigg|\Phi(a,\xi';V_s[r](x,x'))\eta_s(x,x')\bigg|\dd\sigma_{s}(\xi',x')
             \\
             & \qquad +\int_{\R\times \Rdminus}\bigg|\Phi(a,\xi';V_t[r](x,x'))\eta_t(x,x')\bigg|\dd\sigma_{s}(\xi',x')
             \\
             & \leq  L_\Phi \norm{\eta}_{\infty,C_b(\Rddiag)}\bigg(|a|\int_{\Rd}|V_s[r](x,x')|\dd\mu(x') + \int_{\Rd}|V_s[r](x,x')|\int_{\R}|\xi'|\dd\sigma_{s,x'}(\xi')\dd\mu(x')
             \\
             &\qquad + |a|\int_{\Rd}|V_t[r](x,x')|\dd\mu(x') + \int_{\Rd}|V_t[r](x,x')|\int_{\R}|\xi'|\dd\sigma_{s,x'}(\xi')\dd\mu(x')\bigg)
             \\
             &=L_\Phi \norm{\eta}_{\infty,C_b(\Rddiag)}\bigg(|a|\int_{\Rd}|V_s[r](x,x')\dd\mu(x') + \int_{\Rd}|V_s[r](x,x')|m_1[\sigma_{s,x'}]\dd\mu(x')
             \\
             & \qquad + |a|\int_{\Rd}|V_t[r](x,x')|\dd\mu(x') + \int_{\Rd}|V_t[r](x,x')|m_1[\sigma_{s,x'}]\dd\mu(x')\bigg)\\
             & \leq 2L_\Phi \norm{\eta}_{\infty,C_b(\Rddiag)} C_V^{1/2} |a|  + L_\Phi \norm{\eta}_{\infty,C_b(\Rddiag)}C_V^{1/2}\bigg(\left(\int_{\Rd}m_1^2[\sigma_{s,x'}]\dd\mu(x')\right)^{1/2}
             \\
             &\qquad +  \left(\int_{\Rd}m_1^2[\sigma_{s,x'}]\dd\mu(x')\right)^{1/2}\bigg)
\end{align*}
\begin{align*}
           & = 2L_\Phi \norm{\eta}_{\infty,C_b(\Rddiag)} C_V^{1/2} |a|  
            \\
            &\quad + L_\Phi \norm{\eta}_{\infty,C_b(\Rddiag)}C_V^{1/2}\!\!\left[\left(\!\int_{\Rd}\!\left(\int_{\R}|\xi'|\dd\sigma_{s,x'}(\xi')\right)^2\dd\mu(x')\!\right)^{1/2} \!\!\!\!+  \left(\!\int_{\Rd}\!\left(\int_{\R}|\xi'|\dd\sigma_{s,x'}(\xi')\right)^2\dd\mu(x')\!\right)^{1/2}\right]
             \\
             & \leq 2L_\Phi \norm{\eta}_\infty C_V^{1/2} |a|  + 2L_\Phi \norm{\eta}_{\infty,C_b(\Rddiag)}C_V^{1/2}\left(\int_{\Rd}m_2[\sigma_{s,x'}]\dd\mu(x')\right)^{1/2} \ 
             \\
             & = 2L_\Phi \norm{\eta}_\infty C_V^{1/2} |a|  + 2L_\Phi \norm{\eta}_{\infty,C_b(\Rddiag)}C_V^{1/2} m_2^{1/2}[\gamma_s] <\infty\ , 
        \end{align*}
    where in the third inequality we used Cauchy-Schwarz and \eqref{eq:velocity_second_moment}, and the final inequality follows from Jensen's inequality.   Then, since $t\mapsto V_t$ and $t \mapsto \eta_t$ are continuous, by the Dominated Convergence Theorem we have that $II \to 0$ as $t \to s$. The limit $s \to t$ follows from an analogous argument. 
    The remainder of the statement follows from the Cauchy-Lipschitz theory for ODEs.  
\end{proof}
The previous lemma guarantees we can consider, for $\mu$-a.e. $x\in\Rd$, a flow map $f_x[\sigma, r,\eta]$, dependent on $\sigma,r,\eta$, which is the solution to
\begin{align}\label{eq:flow_equation}
    &\frac{d}{dt}f_x[\sigma,r,\eta](t,u) =   \X[\sigma,r,\eta](t,f_x[\sigma,r,\eta](t,u), x)\ , 
    \\
    & f_x[\sigma,r,\eta](0,u) = u \in \R \ , \nonumber
\end{align}
In particular, for an initial time $s>0$, we write $f_x[\sigma,r,\eta]((s,t),u)$, meaning that $f_x[\sigma,r,\eta](s,u) = u$. Before establishing the well-posedness of the nonlinear problem \eqref{eq:vlasov_equation} we need to study some further properties of the flow.

\begin{lemma}\label{lemma:properties_of_flow}
    Let $V:[0,T]\times L^2_\mu(\Rd)  \to \V^{as}(\Rddiag)$ be continuous in time and such that it satisfies \eqref{eq:velocity_second_moment}~-~\eqref{eq:velocity_Lipschitz_second_moment}. Let $\sigma \in C([0,T], \P^{\mu}_2(\R{\times\Rd}))$ with ${\pi_1}_\#\sigma=\gamma$ and $(r,\eta)$ be the solution of \eqref{eq:euler}. The flow map $f_x[\sigma,r,\eta]$ defined in \eqref{eq:flow_equation} satisfies the properties below.
    \begin{enumerate}
        \item Linear growth: for all $T>0$, there exists a constant $\widetilde{C}=C(\eta_0,m_2[\gamma]) > 0$ such that
        \[
        |f_x[\sigma,r,\eta](t,u)| \leq e^{\widetilde{C}T}(1 + |u|),
        \]
        for any $t\in[0,T]$, $u\in\R$, $\mu$-a.e. $x\in\Rd$.
        \item Lipschitz in $u$: there exists a constant $\bar{C}=C(\eta_0,V,\Phi) > 0$ such that for any $t\in[0,T]$ and any  $u,\ell \in \R$
        \[
        |f_x[\sigma,r,\eta](t,u) - f_x[\sigma, r,\eta](t,\ell)| \leq  e^{\bar{C}T}|u-\ell| \ .
        \]
        \item Time continuity: the map $t \mapsto f_x[\sigma,r,\eta](t,\cdot)$ is continuous.
    \end{enumerate}
\end{lemma}
\begin{proof}
    We begin by proving (1). By definition of the flow we have
    \begin{align*}
    f_x[\sigma,r,\eta](t,u) = u + \int_0^t\X[\sigma,r,\eta](s,f_x[\sigma,r,\eta](s,u),x) \dd s  \ . 
    \end{align*}
    By means of~\ref{ass:interp_deg},~\ref{ass:interp_lip} in Definition~\ref{def:admissible_interpolation},
    \begin{align*}
        \big|\X[\sigma, r,\eta] &(s,f_x[\sigma,r,\eta](s,u),x)\big|
        \\
        & \leq \int_{\R\times \Rdminus} \bigg|\Phi(f_x[\sigma,r,\eta](s,u),\xi';V_s[r](x,x'))\eta_s(x,x')\bigg|\dd \sigma_s(\xi',x') 
        \\
        & \leq L_{\Phi}\norm{\eta}_{\infty, C_b(\Rddiag)} C_V^{1/2}\bigg(|f_x[\sigma,r,\eta](s,u)|\! +\! \left(\int_{\Rd}\!\!m_1^2[\sigma_{s,x'}]\dd\mu(x')\right)^{1/2} \!\bigg)
        \\
        & \leq L_{\Phi}\norm{\eta}_{\infty, C_b(\Rddiag)} C_V^{1/2}\bigg(|f_x[\sigma,r,\eta](s,u)| + m_2[\gamma_s]^{1/2}\bigg)
        \\
        &\leq L_{\Phi}\norm{\eta}_{\infty, C_b(\Rddiag)} C_V^{1/2}\bigg(|f_x[\sigma,r,\eta](s,u)| + m_2[\gamma_*]^{1/2}\bigg)\ 
        \\
        & \leq L_\Phi \left(\|\eta_0\|_{\infty} + C_\w \right) C_V^{1/2}\bigg(|f_x[\sigma,r,\eta](s,u)| + m_2[\gamma_*]^{1/2}\bigg)
        \\
        & \leq {\widetilde{C}}(1 + |f_x[\sigma,r,\eta](s,u)|)\ , 
    \end{align*}
   for a constant $\widetilde{C}=C(\eta_0,m_2(\gamma))$, where we used \eqref{eq:norm_eta_bound} in the penultimate inequality. An application of Gr\"onwall's inequality yields
    \begin{align*}
        |f_x[\sigma,r,\eta](t,u)| \leq e^{\widetilde{C}T}(1 + |u|).
    \end{align*}
To prove (2), consider $u,\ell \in \R$ and use \eqref{eq:Lipschitz_velocity_field} from Lemma \ref{lemma:existence_of_flow} in order to have
    \begin{align*}
        |f_x[\sigma, r,\eta](t,u) - f_x[\sigma, r,\eta](t,\ell)|& \leq  |u - \ell| + \int_0^t|\X[\sigma, r,\eta](s,f_x[\sigma, r,\eta](s,u),x)
        \\
        & \qquad \qquad \qquad \qquad - \X[\sigma, r,\eta](s,f_x[\sigma, r,\eta](s,\ell),x)|\dd s
        \\
        & \leq |u-\ell|\\
        &\qquad+ L_\Phi C_V^{1/2} \norm{\eta}_{\infty, C_b(\Rddiag)} \int_0^t|f_x[\sigma, r,\eta](s,\ell) - f_x[\sigma, r,\eta](s,u)| \dd s .
    \end{align*}
    An application of Gr\"onwall's inequality yields
    \begin{align*}
      |f_x[\sigma, r,\eta](t,u) - f_x[\sigma, r,\eta](t,\ell)| &\leq \exp\left(L_\Phi C_V^{1/2} \norm{\eta}_{\infty, C_b(\Rddiag)} T\right)|u-\ell| \ ,
      \\
      & \leq \exp\left(L_\Phi C_V^{1/2} \left(\|\eta_0\|_{\infty} + C_\w \right)T\right)|u-\ell|
      \\
      & \leq e^{\bar{C}T}|u - \ell|
    \end{align*}
for a constant $\bar{C}=C(\eta_0,V,\Phi) > 0$, which concludes the proof of (2). The final statement follows from the definition of the flow and the time continuity of $\X[\sigma, r,\eta]$ proved in Lemma \ref{lemma:existence_of_flow}. 
\end{proof}
The flow map depends on a probability measure $\sigma$ and a given co-evolving graph $\mathcal{G}_t(\mu,\eta_t[r])$. To conclude our argument, we estimate the difference of the flow maps depending on different measures, $\sigma^1,\ \sigma^2$. To this end, let us consider the space $C([0,T], \P^\mu_2(\R\times \Rd))$ equipped with the metric 
\[
\mathcal{D}_{\mu,d}(\nu^1,\nu^2):=\sup_{t\in[0,T]}L^2_\mu d_2(\nu^1_t,\nu^2_t) = \sup_{t\in[0,T]} \left(\int_{\Rd} d^2_2(\nu^1_{t,x'},\nu^2_{t,x'})\dd \mu(x')\right)^{1/2}, 
\]
for $\nu^1,\nu^2 \in C([0,T], \P^\mu_2(\R\times\Rd))$, so that $(C([0,T], \P^\mu_2(\R\times \Rd)), \mathcal{D}_{\mu,d})$ is a complete metric space, see~\cite[Proposition A.0.1]{mikol2024a} for a similar proof. We obtain the following estimate. 
\begin{lemma}\label{lemma:flow_contraction}
    Let $\sigma^i \in C([0,T],\P^\mu_2(\R\times\Rd))$, with ${\pi_1}_\#\sigma^i=\gamma^i$ and $(r,\eta)$ be the solution of \eqref{eq:euler} with initial condition $(r_0,\eta_0)$. Let  $V:[0,T]\times  L^2_\mu(\Rd) \to \V^{as}(\Rddiag)$ be continuous in time satisfying \eqref{eq:velocity_second_moment} and~\eqref{eq:velocity_Lipschitz_second_moment}. Consider two flows $f^i_x[\sigma^i, r, \eta]$, for $\mu$-a.e. $x \in \Rd$, solutions to~\eqref{eq:flow_equation}, for $i =1,2$. Then, it holds 
    \begin{align}\label{eq:sq_flow_bound}
        |f^1_x[\sigma^1,r,\eta](t,u) - f_x^2[\sigma^2,r,\eta](t,u)|^2\leq\bar{C}_TT^2\mathcal{D}_{\mu,d}^2(\sigma^1,\sigma^2),
\end{align}
where $\bar{C}_T=L_\Phi^2\left( \norm{\eta_0}_{\infty}+C_\omega\right)^2 C_V\exp\{ L_\Phi^2 \left( \norm{\eta_0}_{\infty}+C_\omega\right)^2 C_V \ T^2\}$. Consequently, for any measure $\nu \in \P_2^\mu(\R\times\Rd)$, with disintegration $\nu = \int_{\Rd}\nu_x \dd\mu(x)$, we have
    \begin{equation}\label{eq:contraction}
        \D_{\mu,d}(f^1[\sigma^1,r,\eta]_\# \nu, f^2[\sigma^2,r,\eta]_\# \nu) \leq \tilde{C}_TT \D_{\mu,d}(\sigma^1,\sigma^2) \,
    \end{equation}
where $\tilde{C}_T=\bar{C}_T^{\frac{1}{2}}$ from above.
\end{lemma}
\begin{proof}
To improve the readability we use the shorthand notation $f^i_x := f^i_x[\sigma^i,r,\eta]$, for $i=1,2$. Let us start by estimating the difference of the two flows: 
\begin{align*}
    |f^1_x(t,u) - f_x^2(t,u)|  
    =&  \bigg|\int_0^t \int_{\R\times\Rdminus} \Phi(f_x^1(s,u),\xi';V_s[r_s](x,x'))\eta_s(x,x')\dd \sigma^1_s(\xi',x') 
\\
&\qquad - \int_{\R\times\Rdminus} \Phi(f_x^2(s,u),\xi';V_s[r_s](x,x'))\eta_s(x,x')\dd \sigma^2_s(\xi',x') \dd s \bigg|
\\
\leq &\ \int_0^t\bigg|\int_{\R\times\Rdminus}\bigg( \Phi(f_x^1(s,u),\xi';V_s[r](x,x'))\eta_s(x,x')
\\
& \qquad - \Phi(f_x^2(s,u),\xi';V_s[r_s](x,x'))\eta_s(x,x')\bigg)\dd \sigma_s^1(\xi',x')\bigg|  \dd s 
\\
&   + \int_0^t  \bigg|\int_{\R\times\Rdminus}\Phi(f_x^2(s,u),\xi';V_s[r_s](x,x'))\eta_s(x,x')\dd (\sigma_s^1-\sigma^2_s)(\xi',x') \bigg|\dd s 
\\
=: & I + II. 
\end{align*}
We bound the term $I$ analogous calculation to the ones for \eqref{eq:term_1_euler} to obtain
\begin{align*}
    |I|&= \int_0^t\bigg| \int_{\R\times\Rdminus}\!\!\!\bigg( \Phi(f_x^1(s,u),\xi';V_s[r_s](x,x')) \!
    \\
    & \qquad \qquad -\! \Phi(f_x^2(s,u),\xi';V_s[r_s](x,x'))\bigg) \eta_s(x,x') \dd \sigma_s^1(\xi',x')\bigg|  \dd s
\\
    & \leq  L_\Phi \norm{\eta}_{\infty,C_b(\Rddiag)}\int_0^t\int_{\Rd}|V_s[r_s](x,x')||f_x^1(s,u)-f_x^2(s,u)|\dd \mu(x')
    \\
    & \leq L_\Phi \norm{\eta}_{\infty,C_b(\Rddiag)} C_V^{1/2} \int_0^t|f_x^1(s,u)-f^2_x(s,u)|\dd s,
\end{align*}
whence
\[
|I|^2 \leq L_\Phi^2 \norm{\eta}^2_{\infty,C_b(\Rddiag)} C_V \ T\int_0^t|f_x^1(s,u)-f^2_x(s,u)|^2\dd s \ .
\]
On the other hand, for $II$, we have  
\begin{equation}\label{eq:lemma3.3_bound_II}
\begin{split}
|II | &\leq \int_{0}^t\int_{\R^d} \bigg|\int_{\R} \Phi(f_x^2(s,u),\xi';V_s[r_s](x,x'))\eta_s(x,x')\dd(\sigma^1_{s,x'}-\sigma^2_{s,x'})(\xi')\bigg| \dd\mu(x')\dd s
       \\
       & \leq  \int_0^t\int_{\Rd} Lip[\Phi(f_x^2(s,u),\cdot,V_s[r_s](x,x')\eta_s(x,x')]d_1(\sigma^1_{s,x'},\sigma^2_{s,x'})\dd\mu(x') \dd s
       \\
       &\leq L_\Phi \norm{\eta}_{\infty,C_b(\Rddiag)} \int_0^t\int_{\Rd}|V_s[r_s](x,x')|d_1(\sigma^1_{s,x'},\sigma^2_{s,x'})\dd\mu(x') \dd s
       \\
       &\leq L_\Phi \norm{\eta}_{\infty,C_b(\Rddiag)}C_V^{1/2} \int_0^t\left(\int_{\Rd}d^2_1(\sigma^1_{s,x'},\sigma^2_{s,x'})\dd\mu(x')\right)^{1/2} \dd s
       \\
       &\leq L_\Phi \norm{\eta}_{\infty,C_b(\Rddiag)}C_V^{1/2}\int_0^t L^2_\mu d_2(\sigma^1_{s},\sigma^2_{s}) \dd s\\
       &\leq L_\Phi \norm{\eta}_{\infty,C_b(\Rddiag)}C_V^{1/2} T\mathcal{D}_{\mu,d}(\sigma^1,\sigma^2) 
       \end{split}
\end{equation}
where the second and third inequality follow from the definition by duality of the 1-Wasserstein distance and \eqref{eq:lip_constant_bound}, the fourth inequality follows from Cauchy-Schwarz. Squaring both sides and taking the supremum over time yields
\begin{align*}
    |II|^2 &\leq L_\Phi^2 \norm{\eta}^2_{\infty,C_b(\Rddiag)}C_V \mathcal{D}_{\mu,d}^2(\sigma^1,\sigma^2)T^2 \ . 
\end{align*}
Thus, collecting all the terms we get that 
\begin{align*}
    |f^1_x(t,u) - f_x^2(t,u)|^2    
      \leq \  &L_\Phi^2 \norm{\eta}^2_{\infty,C_b(\Rddiag)} C_V \ T\int_0^t |f_x^1(s,u)-f^2_x(s,u)|^2\dd s
    \\
    & + L_\Phi^2 \norm{\eta}^2_{\infty,C_b(\Rddiag)}C_V\mathcal{D}_{\mu,d}^2(\sigma^1,\sigma^2)T^2,
\end{align*}
and an application of Gr\"onwall's inequality yields 
\begin{equation}\label{eq:flow_sq_bound}
\begin{split}
|f^1_x(t,u) - f_x^2(t,u)|^2 &\leq  L_\Phi^2\norm{\eta}^2_{\infty,C_b(\Rddiag)}C_V\mathcal{D}_{\mu,d}^2(\sigma^1,\sigma^2)T^2\exp\{ L_\Phi^2 \norm{\eta}^2_{\infty,C_b(\Rddiag)} C_V \ T^2\}\\
&\leq \bar{C}_TT^2\mathcal{D}_{\mu,d}^2(\sigma^1,\sigma^2),
\end{split}
\end{equation}
where $\bar{C}_T=L_\Phi^2\left( \norm{\eta_0}_{\infty}+C_\omega\right)^2 C_V\exp\{ L_\Phi^2 \left( \norm{\eta_0}_{\infty}+C_\omega\right)^2 C_V \ T^2\}$. The first statement is therefore proven. To obtain the second estimate, consider the transference plan $\Pi_{t,x} = (f_x^1 \times f_x^2)_\# \nu_x$. For $\mu$-a.e. $x \in \Rd$ we have that 
\begin{align*}
    d^2_2({f_x}^1_\#{\nu_x}, {f_x}^2_\#{\nu_x}) &\leq \int_{\R\times \R} |a-b|^2 \dd \Pi_{t,x}(a,b)
    \\
    & \leq \int_{\R}|f_x^1(t,u) - f_x^2(t,u)|^2 \dd \nu_x(u) 
    \\
    & \leq \bar{C}_TT^2\mathcal{D}_{\mu,d}^2(\sigma^1,\sigma^2),
\end{align*}
where the final inequality follows from \eqref{eq:flow_sq_bound}. Integrating with respect to $\mu$, taking the square root on both sides and the supremum in time gives 
\[
\D_{\mu,d} (f^1_\#\nu,f^2_\#\nu) \leq \tilde{C}_TT \mathcal{D}_{\mu,d}(\sigma^1,\sigma^2),
\]
where $\tilde{C}_T=\bar{C}_T^{\frac{1}{2}}$.
\end{proof}
 In order to prove well-posedness of \eqref{eq:vlasov_equation} through its disintegrated version, we first need to study well-posedness for solutions of a linear version of this problem, namely, given $\sigma\in C([0,T];\P_2^\mu(\R\times\Rd))$, 
\begin{equation}\label{eq:disintegrated_linear_vlasov_equation}
     \begin{split}
     &\partial_t \tilde{\sigma}_{t,x} + \partial_{\xi}(\tilde{\sigma}_{t,x}\X[\sigma,\rho,\eta]) = 0,  
     \\
     & \tilde{\sigma}_{0,x} = \bar{\sigma}_x \in \P(\R),
     \end{split}
 \end{equation}
 for $\mu$-a. e. $x\in\Rd$ and $\bar\sigma=\int_{\Rd}\bar{\sigma}_x\dd\mu(x)$. 
We prove existence and uniqueness of solutions to~\eqref{eq:disintegrated_linear_vlasov_equation} by duality, that is, we first consider classical solutions to the transport equation, $\mu\text{-a.e. } x \in \Rd$: 
\begin{equation}\label{eq:dual_linear_vlasov}
    \begin{cases}
    \partial_t \varphi + \partial_\xi \varphi\X[\sigma, r, \eta](t,\xi,x) = 0, \qquad t<T, \, (\xi,x)\in\R\times\Rd
    \\
    \varphi(T,\xi,x) = \psi(\xi,x) \in C^1_c(\R\times \Rd).
    \end{cases}
\end{equation}
The method of characteristics provides the following result; see~\cite[Chapter 3]{evans2022partial} for further details. 
\begin{proposition}
    Let ${\sigma}\in C([0,T],\P^\mu_2(\R\times \Rd))$. The problem \eqref{eq:dual_linear_vlasov} has a unique classical solution  $\varphi(t,u,x) = \psi(f_{x}[\sigma,r,\eta]((0,t), u)) \in C^1_c([0,T]\times \R)$, for $\mu$-a.e. $x \in \Rd$, where $f_x$ is as in~\eqref{eq:flow_equation}. 
\end{proposition}
For a given $\sigma\in C([0,T];\P_2^\mu(\R\times\Rd))$, we can construct explicit distributional solutions for~\eqref{eq:disintegrated_linear_vlasov_equation} as in the following theorem.
\begin{proposition}\label{prop:well_posedness_linear_vlasov}
 Let  $\tilde{\sigma}_{0} \in \P^\mu_2(\R\times \Rd)$ such that  $\tilde{\sigma}_0 = \int_{\Rd}\tilde{\sigma}_{0,x'}\dd\mu(x')$ and $\sigma \in C([0,T], \P^\mu_2(\R \times \Rd))$. Assume $V:[0,T]\times L^2_\mu(\Rd) \to \V^{as}(\Rddiag)$ satisfies \eqref{eq:velocity_second_moment}, \eqref{eq:velocity_Lipschitz_second_moment} and that it is continuous in time. Let $(r,\eta)$ be a solution to~\eqref{eq:euler}. Then, $\tilde{\sigma}_{t,x} = f_x[\sigma,r,\eta](t,x,\cdot)_\# \tilde{\sigma}_{0,x}$ is the unique distributional solution to \eqref{eq:disintegrated_linear_vlasov_equation}, for $\mu$-a.e. $x \in \Rd$, where $f_x$ is defined in~\eqref{eq:flow_equation}.
\end{proposition}
\begin{proof}
First, we notice that by direct computation it can be checked that $\tilde{\sigma}_{t,x} = f_x[\sigma,\rho,\eta](t,\cdot)_\# \tilde{\sigma}_{0,x}$ satisfies the equation in the sense of distributions, i.e.,~\eqref{eq:mu_ae_sol_concept_vlasov}, for $\mu$-a.e. $x \in \Rd$. Secondly, by the linearity of \eqref{eq:disintegrated_linear_vlasov_equation}, we can show uniqueness for solutions with $\tilde{\sigma}_{0} = 0$. By choosing $\varphi(t,\xi,x) = \psi(f_x[\sigma,\rho,\eta]((0,t),\xi))$ as a test function in \eqref{eq:mu_ae_sol_concept_vlasov}, we obtain 
\begin{align*}
    \int_{\R} \varphi(T,\xi,x) \dd \tilde{\sigma}_{T,x}(\xi) = \int_{\R}\psi(\xi,x)\dd \tilde{\sigma}_{T,x}(\xi) = 0 \ , \ \mu\text{-a.e. } x \in\Rd, 
\end{align*}
for all $\psi\in C^1_c(\R\times\Rd)$, whence $\tilde{\sigma}_{T,x}=0$ and we obtain uniqueness.
\end{proof}
We are ready to prove the well-posedness of \eqref{eq:vlasov_equation}. For any $\nu \in \P^\mu_2(\R\times\Rd)$ with disintegration $\nu = \int_{\Rd}\nu_x \dd \mu(x)$, $x \in \supp\mu$ and the solution $(r,\eta)$ of \eqref{eq:euler},  we define the solution map $\S^{\nu,r,\eta}:C([0,T], \P^\mu_2(\R\times \Rd)) \to C([0,T], \P^\mu_2(\R\times \Rd))$ as
\begin{equation}\label{eq:sol_map}
    \S^{\nu,r,\eta}[\sigma](t) :=\int_{\Rd} f_x[\sigma,r,\eta](t,\cdot)_\# \nu_x\dd\mu(x) \ , \quad \sigma \in \P_2^\mu(\R\times \Rd), 
\end{equation}
where $f_x[\sigma,r,\eta]$ is defined in \eqref{eq:flow_equation}. Note that the solution map is well defined due to the previous results. We denote the disintegration with respect to $\mu$ by $\S^{\nu,r,\eta}_x[\sigma](t) := f_x[\sigma,r,\eta](t,\cdot)_\# \nu_x$, for $\mu$-a.e. $x \in \Rd$. We prove well-posedness of the nonlinear problem~\eqref{eq:vlasov_equation} by another fixed-point argument.
\begin{theorem}\label{thm:vlasov_well_posedness}
    Let $V:[0,T]\times L^2_\mu(\Rd) \to \V^{as}(\Rddiag)$ be continuous in time and satisfy \eqref{eq:velocity_second_moment} and \eqref{eq:velocity_Lipschitz_second_moment}. Let $(r,\eta)$ be the solution of \eqref{eq:euler}. Then, there exists a unique solution $\sigma \in C([0,T],\P^\mu_2(\R\times\Rd))$ to \eqref{eq:vlasov_equation} with initial condition $\nu\in \P^\mu_2(\R\times\Rd)$ given by $\sigma = \int_{\Rd}\sigma_{t,x}\dd\mu(x)$ where $\sigma_{t,x}=f_x[\sigma,r,\eta](t,\cdot)_\#\nu_x$ for $\mu$-a.e. $x \in \Rd$.
\end{theorem}
\begin{proof}
    Let us consider the flow $f_x[\sigma,\rho, \eta]$ as defined in \eqref{eq:flow_equation}, for some $\sigma \in C([0,T],\P_2^\mu(\R\times\Rd))$. Note that by Proposition \ref{prop:well_posedness_linear_vlasov}, we have that $\tilde{\sigma}_x = \S_x^{\nu,r,\eta}[\sigma]$ is the unique solution in $C([0,T],\P^\mu_2(\R\times\Rd))$ to the linear problem \eqref{eq:disintegrated_linear_vlasov_equation} 
     with initial datum $\nu \in \P^\mu_2(\R\times\Rd)$.  Thus, to obtain existence and uniqueness of solution to the nonlinear problem \eqref{eq:vlasov_equation} it suffices to show the solution map $\S^{\nu,r,\eta}[\sigma]$ has a unique fixed-point. Using the disintegrated version of the solution map, $\S^{\nu,r,\eta}_x[\sigma](t)$, and Lemma~\ref{lemma:flow_contraction} we have, for any $\sigma^1,\sigma^2 \in C([0,T],\P_2^\mu(\R\times\R^d))$, 
     \[
     \D_{\mu,d}(\S^{\nu,r,\eta}[\sigma^1], \S^{\nu,r,\eta}[\sigma^2]) \leq \tilde{C}_T T\D_{\mu,d}(\sigma^1,\sigma^2) ,
     \]
    where $\tilde{C}_T>0$ is the constant from \eqref{eq:contraction} and it is of the form $\tilde {C}_T = C \exp\{\alpha T^2\}$. Then, for $T < \tilde{C}_T^{-1}$ the solution map is a contraction. By the Banach-fixed-point Theorem we can conclude there exists a unique solution to \eqref{eq:vlasov_equation} in $C([0,T],\P_2^\mu(\R\times \Rd))$ for $T < T^*$, where $T^*$ is such that $\tilde{C}_{T*}{T^*}=1$. The solution can be extended using a standard iteration procedure so that the time interval does not depend on any $T^*$, see, e.g., the proof of Theorem \ref{thm:well-posedness} in~\cite{esposito_mikolas_2024}.
\end{proof}
\begin{theorem}\label{thm:vlasov_dobrushin}
 Assume $V:[0,T]\times L^2_\mu(\Rd) \to \V^{as}(\Rddiag)$ to be continuous in time and to satisfy~\eqref{eq:velocity_second_moment},~\eqref{eq:velocity_Lipschitz_second_moment}. Let $(r,\eta)$ be the solution to \eqref{eq:euler} with initial datum $(r_0, \eta_0)\in L^2_\mu(\Rd) \times C_b(\Rddiag)$ and let  $\sigma^i \in C([0,T], \P_2^\mu(\R\times\Rd))$ be the solutions to the corresponding \eqref{eq:vlasov_equation} with initial conditions $\tilde{\sigma}^i_0\in\P^\mu_2(\R\times\Rd)$ for $i=1,2$. The following stability estimate holds
 \[
 \mathcal{D}_{\mu,d}(\sigma^1,\sigma^2) \leq \sqrt{2}e^{C(T)} L^2_\mu d_2(\tilde{\sigma}^1_0,\tilde{\sigma}^2_0).
 \]
\end{theorem}
\begin{proof}
    Let us write $f^i_x = f^i_x[\sigma^i,r,\eta]$, for $i=1,2$. Then we have that, for $\mu$-a.e. $x\in\Rd$
    \begin{align*}
       d_2^2(\sigma^1_{t,x}, \sigma^2_{t,x}) & = d_2^2(f^1_x[\sigma^1,r,\eta](t,\cdot)_\#\tilde{\sigma}^1_{0,x}, f^2_x[\sigma^2,r,\eta](t,\cdot)_\#\tilde{\sigma}^2_{0,x})
        \\
        & \leq 2 d_2^2(f^1_x(t)_\#\tilde{\sigma}^1_{0,x}, f^2_x(t)_\#\tilde{\sigma}^1_{0,x}) + 2 d_2^2(f^2_x(t)_\#\tilde{\sigma}^1_{0,x}, f^2_x(t)_\#\tilde{\sigma}^2_{0,x})
        \\
        & \leq \mathcal{C}(T) \left(\int_0^t  L^2_\mu d_2({\sigma^1_s}, {\sigma^2_s})\dd s \right)^2 + 2d_2^2(f^2_x(t)_\#\tilde{\sigma}^1_{0,x}, f^2_x(t)_\#\tilde{\sigma}^2_{0,x}) \ , 
    \end{align*}
where {$\mathcal{C}(T)=L_\Phi^2\left( \norm{\eta_0}_{\infty}+C_\omega\right)^2 C_V\exp\{ L_\Phi^2 \left( \norm{\eta_0}_{\infty}+C_\omega\right)^2 C_V \ T^2\}$ is the constant obtained as in \eqref{eq:flow_sq_bound} if we do not take the supremum in time in~\eqref{eq:lemma3.3_bound_II}.} To bound the second term on the right-hand side we proceed as follows.  
Let $\Pi_{0,x} \in \Gamma(\tilde{\sigma}_{0,x}^1,\tilde{\sigma}_{0,x}^2)$ be an optimal transference plan for $\tilde{\sigma}_{0,x}^1$ and $\tilde{\sigma}_{0,x}^2$. Then for $\mu$-a.e. $x\in\Rd$ we have that
\begin{align*}
d_2^2(f^2_x(t)_\#\tilde{\sigma}^1_{0,x}, f^2_x(t)_\#\tilde{\sigma}^2_{0,x}) & \leq \int_{\R\times\R}|f^2_x(t,u) - f^2_x(t,v)|^2\dd \Pi_{0,x}(u,v)
\\
& \leq e^{\bar{C}_T}\int_{\R\times\R}|u-v|^2\dd \Pi_{0,x}(u,v)
\\
& \leq e^{\bar{C}_T}d_2^2(\tilde{\sigma}_{0,x}^1,\tilde{\sigma}_{0,x}^2) \ ,
\end{align*}
where the second inequality follows from (2) of Lemma~\ref{lemma:properties_of_flow} and $\bar{C}=C(\eta_0,V,\Phi)>0$. Summarising, we get 
\[
d_2^2(\sigma^1_{t,x}, \sigma^2_{t,x}) \leq \mathcal{C}(T) \left(\int_0^t  L^2_\mu d_2({\sigma^1_s}, {\sigma^2_s})\dd s \right)^2  + 2e^{\bar{C}_T}d_2^2(\tilde{\sigma}_{0,x}^1,\tilde{\sigma}_{0,x}^2),
\]
whence, by integrating w.r.t. $\mu$ and taking the square root we obtain
\[
L^2_\mu d_2(\sigma^1_t,\sigma^2_t)\le \bar{\mathcal{C}}(T)\int_0^t  L^2_\mu d_2({\sigma^1_s}, {\sigma^2_s})\dd s + \sqrt{2}e^{\bar{C}_T} L^2_\mu d_2(\sigma_0^1,\sigma_0^2),
\]
where $\bar{\mathcal{C}}(T)=\sqrt{\mathcal{C}(T)}$
An application of Gr\"onwall's inequality yields
\[
L^2_\mu d_2(\sigma^1_t,\sigma^2_t)\le\sqrt{2}e^{\bar{C}_T+\bar{\mathcal{C}}(T)T} L^2_\mu d_2(\sigma_0^1,\sigma_0^2),
\]
hence the result.
\end{proof}
\subsection{Contraction in $L^2_\mu d_2$}\label{sec:contraction}

In this section we focus on proving contraction in $L^2_\mu$ for solutions of~\eqref{eq:vlasov_equation}, as this can allow to study the long-time behaviour of the $\mu$-monokinetic solutions, hence the long-time asymptotics for evolutions on graphs. We present the result in the simpler case of $\Phi \equiv \Phi_{upwind}$ and $V$ and $\w$ not depending on $r \in L^2_\mu(\Rd)$, so that the system \eqref{eq:euler} decouples, i.e. the graph is no longer co-evolving but only evolving.

\begin{proposition}\label{prop:upwind_contraction}
    Fix $\Phi\equiv \Phi_{upwind}$. Let $\sigma^i \in C([0,T], \P^\mu_2(\R\times\Rd))$ be $\mu$-monokinetic solutions of \eqref{eq:vlasov_equation}, with initial data $\sigma^i_0 = \delta_{r^i_0(\cdot)} \otimes \mu$, where $(r^i,\eta) $ are solutions to \eqref{eq:euler} starting from $(r^i_0,\bar{\eta}) \in L^2_\mu(\Rd)\times C_b(\Rddiag)$, for $i=1,2$, and $\bar{\eta}>0$. Let $V:[0,T] \to \V^{as}(\Rddiag)$ be an antisymmetric velocity field which is continuous in time and satisfies  \eqref{eq:velocity_second_moment} and \eqref{eq:velocity_Lipschitz_second_moment}, and let $\w:[0,T] \times \Rddiag \to \R$ satisfy \ref{ass:w_continuous}-\ref{ass:omega_symmetric}, and $\omega\ge \omega_*>0$. Assume, 
    \begin{equation}\label{eq:lambda_upwind}
       \lambda := \inf_{t\in[0,T]}\inf_{x \in \Rd}\int_{\Rd}V_t(x,x')\eta_t(x,x')\dd\mu(x') > 0.
    \end{equation}
    Then, for any $t\in[0,T]$,
    \[
    L^2_\mu d_2(\sigma^1_t, \sigma^2_t) \leq e^{- 2\lambda t}L^2_\mu d_2(\sigma^1_0, \sigma^2_0)\ . 
    \]
\end{proposition}
\begin{proof}
First, we notice that $\eta_t>0$ since $\eta_0=\bar\eta>0$ and $\omega\ge \omega_*>0$, by means of the explicit solution form of $\eta$ given in~\eqref{eq:explicit_solution_eta} (see Lemma~\ref{lemma:infty_eta_positive}). Let $\sigma^i_{t,x} = \delta_{r^i_t(x)}\in \P_2(\R)$ be the disintegration of $\sigma^i$ with respect to $\mu$ for $i=1,2$. Then, for $\mu$-a.e. $x \in\Rd$, $\Pi_{t,x} = \delta_{r^1_t(x)}\otimes \delta_{r^2_t(x)}$ is the optimal transference plan between $\sigma^1_{t,x}$ and $\sigma^2_{t,x}$, so that 
    \[
    \left(L^2_\mu d_2(\sigma^1_t,\sigma^2_t)\right)^2 = \int_{\Rd}d^2_2(\sigma^1_{t,x},\sigma^2_{t,x})\dd\mu(x) = \int_{\Rd}|r^1_t(x) - r^2_t(x)|^2\dd\mu(x)\ .
    \]
    Since $r^1,r^2 \in \AC([0,T],L^2_\mu(\Rd))$ are solutions to~\eqref{eq:euler}, by exploiting the Dominated Convergence Theorem, we obtain the following estimate:
        \begin{align*}
        \frac{d}{dt} &\left(L^2_\mu d_2(\sigma^1_t,\sigma^2_t)\right)^2\\
        & =\int_{\Rd}\frac{d}{dt}|r_t^1(x) - r_t^2(x)|^2\dd\mu(x)
        \\
         &= \ -2\int_{\Rd}(r^1_t(x) - r^2_t(x))
        \\
         &\qquad \times  \bigg(r^1_t(x) \int_{\Rdminus}V^+_t(x,x')\eta_t(x,x') \dd \mu(x') - \int_{\Rd}r^1_t(x') V^-_t(x,x') \eta_t(x,x') \dd\mu(x')
        \\
        &  \qquad \qquad - r^2_t(x) \int_{\Rdminus}V^+_t(x,x')\eta_t(x,x') \dd \mu(x') + \int_{\Rd}r^2_t(x') V^-_t(x,x') \eta_t(x,x') \dd\mu(x')\bigg)
        \\
        &=  -2\int_{\Rd}(r^1_t(x) - r^2_t(x))^2\int_{\Rdminus}V^+_t(x,x')\eta_t(x,x')\dd\mu(x')\dd\mu(x)
        \\
        & \quad+ 2 \int_{\Rd}\int_{\Rdminus}(r^1_t(x) - r^2_t(x))(r^1_t(x') - r^2_t(x'))V^-_t(x,x')\eta_t(x,x') \dd\mu(x') \dd\mu(x)\\
         &\leq  
        -2\int_{\Rd}(r^1_t(x) - r^2_t(x))^2\int_{\Rdminus}V^+_t(x,x')\eta_t(x,x')\dd\mu(x')\dd\mu(x)
        \\
        & \quad+  \int_{\Rd}\int_{\Rdminus}[(r^1_t(x) - r^2_t(x))^2 + (r^1_t(x') - r^2_t(x'))^2]V^-_t(x,x')\eta_t(x,x') \dd\mu(x') \dd\mu(x)
\end{align*}
\begin{align*}
        &= -2\int_{\Rd}(r^1_t(x) - r^2_t(x))^2\int_{\Rdminus}V^+_t(x,x')\eta_t(x,x')\dd\mu(x')\dd\mu(x)
        \\
         &\quad +\int_{\Rd}(r^1_t(x) - r^2_t(x))^2\int_{\Rdminus}V^-_t(x,x')\eta_t(x,x') \dd\mu(x') \dd\mu(x) \\
        &  
        \quad+ \int_{\Rd}(r^1_t(x) - r^2_t(x))^2\int_{\Rdminus}V^+_t(x,x')\eta_t(x,x') \dd\mu(x') \dd\mu(x)
        \\
         &= -\int_{\Rd}(r^1_t(x) - r^2_t(x))^2\int_{\Rdminus}V_t(x,x') \eta_t(x,x')\dd\mu(x')\dd\mu(x)
        \\
         &\leq -\lambda \left(L^2_\mu d_2(\sigma^1_t, \sigma^2_t)\right)^2 \ , 
    \end{align*}
    where in the third equality we have used the antisymmetry of $V$ and the symmetry of $\eta$. An application of Gr\"onwall's inequality concludes the proof. 
\end{proof}
Condition~\eqref{eq:lambda_upwind} does not seem to be reasonable since, even for fully connected graphs, i.e. with $\eta>0$ for all $t \geq0$, it can only be satisfied if the velocity is positive for all vertices for any time --- this means that all vertices must \say{give mass} to, at least, one neighbour. It is, therefore, difficult to find examples that fulfil this condition and that are of the form $V(x,x') = -\onabla \phi(x,x') = -(\phi(x') - \phi(x))$ for some suitable, potentially solution dependent, function $\phi$. This motivates the alternative approach presented in Section~\ref{sec:long_time_behaviour} to obtain a possible long-time behaviour of solutions to~\eqref{eq:euler}. Providing a more reasonable condition in place of~\eqref{eq:lambda_upwind} is then an open problem.

\section{Long-time behaviour for pointwise and monotonic velocities}\label{sec:long_time_behaviour}

In this section we consider a specific class of velocities on the graph in order to define a certain order for the flow, helpful to establish a possible long-time behaviour for solutions of~\eqref{eq:euler}, for $\Phi\equiv\Phi_{upwind}$. We highlight two main features. First, velocities are \textit{pointwise}, that is $V_t[r](x,x') = V_t(r_t(x),r_t(x'))$ for a solution $(r,\eta)$ to \eqref{eq:euler}. Second, they are \textit{monotonic}, intuitively meaning that the sign of the velocity, hence the flow, is determined by the amount of mass at the vertices.

\begin{definition}\label{def:monotonic_velocity}
    We say that a pointwise velocity \(V: [0, T] \times L_\mu^\infty(\mathbb{R}^d) \to \V^{as}({\Rddiag})\) is \textit{monotonic} if, for \(r \in L_\mu^\infty(\mathbb{R}^d)\),  any \(t \geq 0\), and \(\mu\)-a.e. \(x, x' \in \mathbb{R}^d\) we have that:
\begin{itemize}
    \item  \(r(x) > r(y)\), then \(V_t^+(r(x), r(x')) > V_t^+(r(y), r(x'))\) ,
    \item  \(r(x) < r(y)\), then \(V_t^-(r(x), r(x')) > V_t^-(r(y), r(x'))\) ,
    \item  \(r(x) = r(y)\), then \(V_t(r(x), r(x')) =V_t(r(y),r(x'))\) ,
    \item  \(r(x) = r(x')\), then \(V_t(r(x), r(x')) =0\).
\end{itemize}
\end{definition} 

 Let us explain the meaning of the conditions in the previous definition. The first condition can be interpreted as saying that, if the mass is greater at a vertex $x\in\Rd$ than $y\in \Rd$, then the outflow from vertex $x$ to its neighbours is greater than the outflow from $y$. Conversely, if the mass at vertex $x$ is lower than the mass at vertex $y$, the inflow to $x$ from its neighbours should be greater than the inflow to $y$. The rest of the conditions follow the same type of reasoning. This assumption on the velocity is similar on the upwind structure of~\eqref{eq:euler}, i.e., $\Phi\equiv \Phi_{upwind}$, although we now require that the direction depends on the quantity of mass. The pointwise nature of the velocities requires we should work in the $L^\infty_\mu$ setting. Furthermore, we stress that monotonic velocities ensure that the mass on the graph diffuses as we shall show in Section~\ref{subsec:long_time_behaviour}. 

Our main example for such velocities is given by 
\begin{equation}\label{eq:pointwise_velocity_example}
    V_t(r_t(x),r_t(x')) = -\onabla \alpha(r_t)(x,x') =  \alpha[r_t(x)] - \alpha[r_t(x')]\ ,
\end{equation}
where $\alpha:{\R}\to\R$ is monotonically increasing and bounded such as $\alpha(x) = \frac{1}{1+ e^{-x}}$ for $x\in\Rd$. We remark that this type of velocity is different from the ones considered in the previous sections, which came from interaction and potential energies.

\subsection{$L^\infty$ theory for \eqref{eq:euler}}\label{subsec:well_posedness_infty}
Let us start by briefly adapting the theory developed in Section~\ref{sec:euler-co-ncl} to the $L^\infty_\mu$-setting. Focusing on the dynamics for $\eta$, the function $\w:[0,T]\times \Linfty \times \Rddiag \to \R$ will now satisfy the following version of assumption \ref{ass:w_continuous}-\ref{ass:omega_symmetric}. 
\begin{enumerate}[label=$(\tilde{\bm{\w}}\arabic*)$]
    \item the map $(x,y)\in\Rddiag  \mapsto \w_t[\cdot](\cdot,x,y)$ is continuous and $(t\mapsto\w_t[\cdot](\cdot,\cdot))\in L^1([0,T])$;\label{ass:infty_w_continuous}
    \item {there exists a constant $L_\w\geq 0$ such that for any $f,g \in \Linfty$ }\label{ass:infty_omega_lip}
     \begin{align*}
         \sup_{t \in [0,T]}\sup_{(x,y) \in \Rddiag} | \w_t[f](x,y) - \w_t[g](x,y)| & \leq L_\w\norm{f - g}_{\Linfty}; 
     \end{align*}
     \item $\w$ is bounded, that is, there exists a constant $\tilde{C}_{\w}>0$ such that \label{ass:infty_omega_bounded}
     \begin{equation*}
    \sup_{t \in [0,T]}\sup_{r \in \Linfty}\sup_{(x,y) \in \Rddiag}\big|\w_t[r](x,y)\big| \le \tilde{C}_{\w}.
\end{equation*}
\item The map $(x,y) \mapsto \w_\cdot[\cdot](x,y)$ is symmetric. \label{ass:infty_omega_symmetric}
\end{enumerate}
Although assumption \ref{ass:infty_omega_symmetric} is not needed to obtain the well-posedness, it is important both to obtain the form \eqref{eq:small_flux} for the divergence of the flux as well as to be able to apply the result in \cite[Proposition 5.2]{Esposito_on_a_class}. Our guiding example for the function $\omega$ is as in~\eqref{eq:omega_example}. Solutions of \eqref{eq:euler} in the $L^\infty_\mu$-setting are defined in Definition~\ref{def:sol_to_euler} for $p=\infty$.
Well-posedness of~\eqref{eq:euler} in the $L^\infty_\mu$-setting can be proven by following the same strategy as in Section~\ref{sec:euler-co-ncl}, by exploiting the Banach fixed-point theorem, with $\mu\in \mathcal{P}(\Rd)$. Below we provide the statements without proofs, cf.~\cite{mikol2024a} for further details.
\begin{proposition}\label{prop:infty_apriori_props}
   Let $\Phi$ be a jointly antisymmetric admissible flux interpolation, $r_0 \in L^\infty_{\mu}(\Rd)$, $\eta_0\in C_b(\Rddiag)$. Assume $\w:[0,T]\times L^\infty_\mu(\R^d)\times \Rddiag  \to  \R$ is such that the map $(x,y)\in\Rddiag  \mapsto \w_t[\cdot](x,y)$ is continuous and that, for a constant $C_{\w}>0$,
\begin{equation}\label{eq:w_integrable_time}
       \int_0^T \sup_{r \in L^\infty_\mu(\R^d)}\sup_{(x,y) \in \Rddiag}|\w_s[r](x,y)| \dd s\le C_{\w}\ .
\end{equation}
Suppose that $V:[0, T]\times L^\infty_\mu(\R^d) \rightarrow \mathcal{V}^{\mathrm{as}}(\Rddiag)$ is pointwise and satisfies  
    \begin{equation}\label{eq:infty_V_bound}
        \sup_{t\in[0,T]} \sup_{r \in \Linfty}\sup_{(x,x') \in \Rddiag} |V_t(r_t(x),r_t(x'))| \leq C_{V}\ , 
    \end{equation}
 for some $C_{V}>0$. For a pair $(r,\eta):[0,T]\to L^\infty_\mu(\Rd)\times C_b(\Rddiag)$ satisfying~\eqref{eq:sol_r,eta}, the following properties hold.
 \begin{enumerate}
     \item For $\mu$-a.e. $x \in \Rd$ and every $(x,y)\in\Rddiag$, the maps 
     \begin{align*}
         &t \mapsto\int_{\Rdminus}\Phi(r_t(x),r_t(y); V_s(r_s(x),r_s(x')))\eta_s(x,x)\dd\mu(y) \in L^1([0,T])\ ,
         \\
         & t\mapsto \w_t[r](x,y) - \eta_t(x,y) \in L^1([0,T])\  ,
     \end{align*}
and $r \in L^\infty([0,T];L^\infty_\mu(\Rd)), \eta\in L^\infty([0,T],C_b(\Rddiag))$.
\item $r \in AC([0,T];L^\infty_\mu(\Rd))$ and $\eta \in AC([0,T];C_b(\Rddiag))$.
 \end{enumerate}
\end{proposition}
\begin{theorem}\label{thm:well_posedness_euler_infty}
    Let $\Phi$ be a jointly antisymmetric admissible flux interpolation as in Definition~\ref{def:admissible_interpolation} and consider an initial datum $(r^0, \eta^0) \in L^\infty_\mu(\Rd)\times C_b(\Rddiag)$. Assume that $V:[0,T]\times L^\infty_\mu(\Rd)  \to \V^{as}(\Rddiag)$ satisfies \eqref{eq:infty_V_bound} and, for any $g,\tilde{g}\in L^\infty_\mu(\Rd)$,
\begin{equation}\label{eq:infty_velocity_Lipschitz_second_moment}
    \sup_{t\in[0,T]} \sup_{(x,x')\in \Rddiag}  |V_t(g_t(x),g_t(x')) - V_t(\tilde{g}_t(x),\tilde{g}_t(x'))| \leq L_{V}\norm{g-\tilde{g}}_{L^\infty_\mu(\Rd)},     
\end{equation}
for a constant $ L_{V}>0$. Assume that $\w:[0,T]\times L^\infty_\mu(\Rd) \times \Rddiag \to \R$ satisfies \ref{ass:infty_w_continuous}-\ref{ass:infty_omega_symmetric}. Then, there exists a solution $(r, \eta)$ to \eqref{eq:euler} with initial data $(r^0,\eta^0)$, in the sense of Definition \ref{def:sol_to_euler}, for $p=\infty$. 
\end{theorem}
\subsection{Long-time behaviour}\label{subsec:long_time_behaviour}
Let us now focus on the long-time asymptotics for~\eqref{eq:euler}. We assume the vertices of the graph belong to a compact space, i.e. $\supp \mu = K \subset\Rd$ for $K$ compact. Analogously to the previous section, we use the following notation $\Kdiag= \{(x,y) \in K^2\ | \ x \neq y\}$. The class of velocities we have in mind is as follows. 
\begin{assumption}\label{ass:V_differentiable}
Assume that  $V:[0,\infty) \times \LinftyK \to \V^{as}(\Kdiag)$ is continuous in time and is of the form 
\[
V(r_t(x),r_t(x')) = \alpha(r_t(x)) - \alpha(r_t(x')),
\]
where $\alpha \in C^1_b(\mathbb{R})$ is increasing on compact sets, i.e. 
\[
\alpha'(y)=\frac{d}{dy}\alpha(y) >c>0\ , \qquad \mbox{for } y \in \Omega\subset\mathbb{R} \mbox{ compact },
\]
and a constant $c$. We denote by ${\alpha'_*} := \inf_{y \in \Omega} \alpha'(y)$.
\end{assumption}
Before proving the main result of this section we recall some technical results --- the first three are a direct adaptation of \cite[Theorem 2.1]{Nastassia_Long_Time}, \cite[Theorem 2.2]{Nastassia_Long_Time} and \cite[Lemma 2.3]{Nastassia_Long_Time}, respectively (see also \cite{berliocchi1973integrandes} and \cite{Danskin} for the original reference of Theorem \ref{thm:scorza} and  Theorem \ref{thm:danskin}, respectively). 

\begin{theorem}[Scorza-Dragoni]\label{thm:scorza}
Let $\Omega \subset \mathbb{R}^d$ be a locally compact set and $f : \mathbb{R}_+ \times \Omega \to \R$ be such that $x \in \Omega \mapsto f(t,x) \in \R$ is $\mu$-measurable for each $t \geq 0$, and $t \in \mathbb{R}_+ \mapsto f(t,x) \in \R$ is continuous for $\mu$-almost every $x \in \Omega$. Then, for every $\varepsilon > 0$, there exists a compact set $\Omega_\varepsilon \subset \Omega$ satisfying $\mu(\Omega \setminus \Omega_\varepsilon) < \varepsilon$ and such that the restricted map $f : \mathbb{R}_+ \times \Omega_\varepsilon \to \R$ is continuous.
\end{theorem}

\begin{theorem}[Danskin]\label{thm:danskin}
Let $\Omega \subset \mathbb{R}^d$ be a compact set and $f : \mathbb{R}_+ \times \Omega \to \mathbb{R}$ be a continuous function such that $t \in \mathbb{R}_+ \mapsto f(t,x) \in \mathbb{R}$ is differentiable for all $x \in \Omega$. Then, the application $g : t \in \mathbb{R}_+ \mapsto \max_{x \in \Omega} f(t,x) \in \mathbb{R}$ is differentiable $\mathscr{L}^1$-almost everywhere, with
\[
\frac{d}{dt} g(t) = \max_{x \in \overline{\Omega}(t)} \partial_t f(t,x)
\]
for $\mathscr{L}^1$-almost every $t \geq 0$, where we introduced the notation $\overline{\Omega}(t) := \arg\max_{x \in \Omega} f(t,x)$.
\end{theorem}
\begin{lemma}[Interior estimates for supremums]\label{lemma:quant_int_estimates}
Let $\Omega \subset \mathbb{R}^d$ be a compact set and $f \in L^\infty_\mu(\Omega, \mathbb{R}^d)$. Then, for every $\delta > 0$, there exists $\varepsilon > 0$ such that
\begin{equation}\label{eq:interior_estimates}
\esssup_{x \in \Omega}f(x) - \delta \leq \esssup_{x \in \Omega^\varepsilon}f(x) \leq \esssup_{x \in \Omega} f(x),
\end{equation}
whenever $\Omega_\varepsilon \subset \Omega$ is a measurable set satisfying $\mu(\Omega \setminus \Omega_\varepsilon) < \varepsilon$. In particular, it holds that
\[
\esssup_{x \in \Omega^\varepsilon}f(x) \xrightarrow{\varepsilon \to 0^+} \esssup_{x \in \Omega}f(x),
\]
for each family of sets  {$(\Omega_\varepsilon)_{\varepsilon>0} \subset \mathrm{P}(\Omega)$} satisfying these properties.
\end{lemma}
\begin{proof}
    The second inequality in \eqref{eq:interior_estimates} is satisfied for any subset of $\Omega$. For the first inequality, by definition of the essential supremum of a set, for every $\delta>0$ there exists a set $\Omega_\delta \subset \Omega$ of positive measure such that 
    \[
    f(x) > \esssup_{x\in \Omega} f(x) - \delta 
    \]
    for $\mu$-a.e. $x \in \Omega_\delta$. Then, by choosing $\varepsilon>0$ satisfying $\varepsilon<\mu(\Omega_\delta)$ and considering any closed set $\Omega_\varepsilon \subset \Omega$ with $\mu(\Omega \backslash \Omega_\varepsilon) < \varepsilon$, it necessarily holds that $\mu(\Omega_\delta \cap \Omega_\varepsilon) >0$. Indeed assuming the measure is $0$, by contradiction, one would have 
    \[
    \mu(\Omega_\delta \cup \Omega_\varepsilon) = \mu(\Omega_\delta) + \mu(\Omega_\varepsilon)> \mu(\Omega)\ ,
    \]
    which is a contradiction since both sets are subsets of $\Omega$. Hence, by construction we have 
     \[
    f(x) > \esssup_{x\in \Omega} f(x) - \delta\ , 
    \]
    for $\mu$-a.e. $x \in \Omega_\delta \cap\Omega_{\varepsilon}$, which concludes the proof. 
\end{proof}

We will also resort to the following comparison results.
\begin{lemma}\label{lemma:generalized_gronwall}
    Let $\phi:\R \to \R$ be a locally-Lipschitz function. Assume that $f\in C(\R^+,\R)$ solves the inequality $f(t) \leq f(0) + \int_0^t\phi(f(s))\dd s$ with $f(0)=a \in \R$ and $g\in C(\R^+,\R)$ solves the equation $g(t) = g(0) + \int_0^t\phi(g(s))\dd s$ with $f(0)\leq g(0)$. Then, for $t\geq0$,
    $f(t) \leq g(t)$.
\end{lemma}
\begin{proof}
    Let $h(t):= f(t) - g(t)$ and, by contradiction, assume there is $b > 0$ such that $h(b) > 0$. By continuity of $h$, Bolzano's theorem implies that there must be a $c \in {[0},b)$ such that $h(t) \ge 0$ for $t \in (c,b]$ and $h(c) = 0$. Then, for $t \in (c,b]$, we have 
    \[
    h(t) \le h(c) + \int_c^t\phi(f(s)) - \phi(g(s))\dd s \leq h(c) + L\int_c^t(f(s) - g(s))\dd s \leq h(c) + L \int_c^t h(s)\dd s ,
    \]
    where, for $t\in(c,b]$, $L>0$ is the local-Lipschitz constant of $\phi$. Then, by Gr\"onwall's inequality 
    \[
    h(b) \leq h(c)e^{L(b-c)} = 0 \ ,
    \]
    which gives the contradiction. 
\end{proof}
The opposite inequality holds and can be proven in a similar way.
\begin{lemma}\label{lemma:lower_diff_ineq}
    Let $\phi:\R \to \R$ be a locally-Lipschitz function. Assume that $f\in C(\R^+,\R)$ solves the inequality $f(t) \geq f(0) + \int_0^t\phi(f(s))\dd s$ with $f(0)=a \in \R$ and $g\in C(\R^+,\R)$ solves the equation $g(t) = g(0) 
 + \int_0^t\phi(g(s))\dd s$ with $f(0)\geq g(0)$. Then for, $t\geq0$, $f(t) \geq g(t)$
\end{lemma}
Throughout this section, we consider \eqref{eq:euler} on the time domain $[0,\infty)$. The well-posedness follows upon adapting assumptions \ref{ass:infty_w_continuous}-\ref{ass:infty_omega_bounded}, \eqref{eq:w_integrable_time}, \eqref{eq:infty_V_bound} and \eqref{eq:infty_velocity_Lipschitz_second_moment} to hold on $[0, \infty)$, and noticing the solution can be extended to $[0,\infty)$ in view of the bounds on the respective norms. To obtain our long-time behaviour result we need to impose some further conditions on $\eta$. First, let us recall that we can obtain an explicit representation for $\eta$ given by
\begin{equation}\label{eq:eta_explicit}
    \eta_t(x,y) = e^{-t}\eta_0(x,y) + \int_0^t e^{-(t-s)}\w_s[r](x,y) \dd s\ .
\end{equation}
We refer the reader to \cite[Lemma A.0.1]{mikol2024a} for a derivation of \eqref{eq:eta_explicit}. Using \eqref{eq:eta_explicit}, we can obtain positivity preservation for the edge-weight function $\eta$ as well as a sharper bound on its supremum norm which we present in the following two results.
\begin{lemma}\label{lemma:infty_eta_positive}
    Assume that $(r,\eta) \in AC([0,\infty), \Linfty)\times AC([0,\infty), C_b(\Kdiag))$ is the solution of \eqref{eq:euler} with initial datum $(r_0,\eta_0) \in (\Linfty,C_b(\Kdiag))$. If, in addition to assumptions \ref{ass:infty_w_continuous}-\ref{ass:infty_omega_bounded}, we impose
    \begin{equation}\label{eq:w_lower_bound}
    \w_*:=\inf_{t \in[0,\infty)} \inf_{(x,x') \in \Kdiag} \inf_{r \in \Linfty\backslash\{0\}}\w_t[r](x,x') > 0\ ,        
    \end{equation}
    and $\eta_0>0$,
     then $\eta_t >0$ for $t \in [0,\infty)$. Furthermore, if $\eta^0 \in C_b(\Kdiag)$ is symmetric, and $\w$ satisfies assumption \ref{ass:infty_omega_symmetric} for any $t\in[0,\infty)$ and $r\in\Linfty$, $\eta_t$ is symmetric for $t \in [0,\infty)$. 
\end{lemma}
\begin{proof}
    Using the representation for $\eta$ in \eqref{eq:eta_explicit} we have, for any $(x,x') \in \Kdiag$ and $t\geq 0$,
    \begin{align*}
        \eta_t(x,x') & = e^{-t }\left(\eta_0(x, x')+\int_0^t e^{s} \omega_s\left[r\right](x, x') \mathrm{d} s\right)\
        \\
        & \geq \inf_{(x,x')\in\Kdiag}\eta_0(x,x')e^{-t} + \w_*(1-e^{-t}) > 0
    \end{align*}
    by \eqref{eq:w_lower_bound} and $\eta_0>0$. The claim about symmetry follows directly from the representation \eqref{eq:eta_explicit}.
\end{proof}

\begin{lemma}
    Let $(r,\eta) \in AC([0,T], \Linfty)\times AC([0,T], C_b(\Kdiag))$ solve \eqref{eq:euler} with initial datum $(r_0,\eta_0) \in (\Linfty,C_b(\Kdiag))$, where $\w$ satisfies \ref{ass:infty_w_continuous}-\ref{ass:infty_omega_symmetric}. Then 
    \begin{equation}
        \norm{\eta_t}_{\infty,C_b(\Kdiag)} \leq \norm{\eta_0}_{\infty,C_b(\Kdiag)} + C_{\w}\ . 
    \end{equation}
\end{lemma}
\begin{proof}
The proof follows directly from the representation \eqref{eq:eta_explicit}. 
\end{proof}

In what follows we will use the notation 
\[
\eta_* := \inf_{t \in[0,\infty)}\inf_{(x,x')\in\Kdiag}\eta_t(x,x') \ ,
\]
and note that we have $\eta_*>0$ in view of Lemma \ref{lemma:infty_eta_positive}. The following theorem holds. 

\begin{theorem}\label{thm:rate_of_sup_norm}
    Fix $\Phi \equiv\Phi_{upwind}$. Let $(r,\eta)\in AC(\R^+,\LinftyK) \times AC(\R^+,C_b(\Kdiag))$ be a solution of \eqref{eq:euler} with a pointwise monotonic velocity field $V:\R^+\times \LinftyK \to \V^{as}(\Kdiag)$ satisfying  \eqref{eq:infty_V_bound}, \eqref{eq:infty_velocity_Lipschitz_second_moment}, and Assumption \ref{ass:V_differentiable}. Assume $\w:\R^+\times\LinftyK\times \Kdiag \to \R$ satisfies \ref{ass:infty_w_continuous}-\ref{ass:infty_omega_symmetric}, and \eqref{eq:w_lower_bound}.    
    Suppose $r_0 \geq 0$ and $\int_{K}r_0(x)\dd\mu(x) = M$, and that $\eta^0 \in C_b(\Kdiag)$ is symmetric and $\eta_0 >0$. Then we have
    \begin{subequations}
    \begin{equation}\label{eq:uniform_boundedness_r}
    \|r\|_{\infty,L^\infty_\mu(K)}\le \|r_0\|_{L^\infty_\mu(K)},   
    \end{equation}
    and, for any $t\ge0$,
    \begin{equation}\label{eq:max_upper_bound}
    \norm{r_t}_{\LinftyK} \leq \frac{\norm{r_0}_{\LinftyK}\alpha'_{*,0}\eta_{*,0}Me^{\alpha'_{*,0}\eta_*Mt}}{\alpha'_{*,0}\eta_*\left[\norm{r_0}_{\LinftyK}\mu(K)(e^{\alpha'_{*,0}\eta_*Mt} - 1) + M\right]},
    \end{equation}
        \end{subequations}
        {where $\alpha'_{*,0}$ indicates that $\alpha'_*$ depends on $r_0$.}
\end{theorem}
\begin{proof}
     We begin by proving~\eqref{eq:uniform_boundedness_r} which is needed for the second result. Let us note that, given our assumptions on $\w$ and $\eta^0$, we have that $\eta_t$ is positive and symmetric for $t \geq 0$. Thus, since we are restricting to the upwind interpolation, by \cite[Proposition 5.2]{Esposito_on_a_class}, \eqref{eq:euler} is non-negativity preserving on $t\geq 0$. Next, fix $\varepsilon>0$ and observe that by Theorem \ref{thm:scorza} there exists a compact set $K^\varepsilon \subset K$ such that the restriction $r:\R^+ \times K^\varepsilon\to \R$ is continuous and $\mu(\KminusKeps) < \varepsilon$. For all $t \geq 0$, we define the family of restricted supremum norms 
    \[
    L_\varepsilon(t) := \max_{x \in K^\varepsilon} |r_t(x)| \ .
    \]
Since $(r,\eta)$ is the solution of \eqref{eq:euler}, we have $r \in AC([0,\infty),L^\infty_\mu(\Rd))$. As any absolutely continuous curve can be reparametrized in time to be Lipschitz continuous (see, e.g.,~\cite[Box 5.1]{santambrogio2015optimal}), the mapping $L_\varepsilon(t)$ is Lipschitz, being the pointwise maximum of a family of equi-Lipschitz continuous functions. In particular, it is differentiable $\mathscr{L}^1$-almost everywhere by Rademacher's theorem. Let us denote by $\mathbb{V}_\varepsilon(t) := \argmax_{x \in K^\varepsilon}r_t(x)$. Then, by Theorem \ref{thm:danskin} we have 
\[
\partial_t L_\varepsilon(t) = \max_{x \in \mathbb{V}_\varepsilon(t) }\partial_t r_t(x) \ . 
\]
Thus, taking an arbitrary vertex $x \in \mathbb{V}_\varepsilon(t)$, we have 
\begin{align}\label{eq:r_t_max_dynamics}
    \partial_t r_t(x)  =  &\  -r_t(x) \int_{\Kepsminus}V_t^+(r_t(x),r_t(x'))\eta_t(x,x')\dd\mu(x') \nonumber
    \\
    &\  -r_t(x) \int_{\KminusKeps}V_t^+(r_t(x),r_t(x'))\eta_t(x,x')\dd\mu(x') 
    \\
    &\ + \int_{\Kepsminus}r_t(x')V_t^-(r_t(x),r_t(x'))\eta_t(x,x') \dd\mu(x')\nonumber
    \\
    & \ + \int_{\KminusKeps}r_t(x')V_t^-(r_t(x),r_t(x'))\eta_t(x,x') \dd\mu(x')\ . \nonumber
\end{align}
For any $x \in \mathbb{V}_\varepsilon(t)$, it holds 
\[
\int_{\Kepsminus} r_t(x')V_t^-(r_t(x),r_t(x')) \eta_t(x,x') \dd \mu(x') = 0 \ , 
\]
since $V_t^-(r_t(x),r_t(x')) = \big(\alpha(r_t(x')) - \alpha(r_t(x))\big)_{+}=0$, $x \in \mathbb{V}_\varepsilon(t)$, $x'\in \Kepsminus$ and $\alpha$ is monotonically increasing. Furthermore, since $\int_{\KminusKeps}V_t^+(r_t(x),r_t(x'))\eta_t(x,x')\dd\mu(x') \geq 0$, by $r_t \geq0$
 we infer%
\begin{align*}
    \partial_t r_t(x)   \leq & \ -r_t(x) \int_{\Kepsminus}V_t^+(r_t(x),r_t(x'))\eta_t(x,x')\dd\mu(x')
    \\
    & \ + \int_{\KminusKeps}r_t(x')V_t^-(r_t(x),r_t(x'))\eta_t(x,x') \dd\mu(x')\ .
\end{align*}
In particular, recalling that for any $x \in \mathbb{V}_\varepsilon(t)$ we have $r_t(x) = L_\varepsilon(t)$ we can rewrite
\begin{equation}\label{eq:differential_inequality_L_eps}
\begin{split}
    \partial_t L_\varepsilon(t)   &\leq   \ -L_\varepsilon(t) \int_{\Kepsminus}V_t^+(L_\varepsilon(t),r_t(x'))\eta_t(x,x')\dd\mu(x')
    \\
    & \quad + \int_{\KminusKeps}r_t(x')V_t^-(L_\varepsilon(t),r_t(x'))\eta_t(x,x') \dd\mu(x')\ 
    \\
    & \leq C_V\norm{r_t}_{\LinftyK}\norm{\eta}_{\infty,C_b(\Kdiag)}\varepsilon \ ,
\end{split}
\end{equation}
where we used the boundedness of $V$ and $\eta$ and that $\mu(\KminusKeps) < \varepsilon$. Integrating with respect to time yields
\begin{align*}
     L_\varepsilon(t) &\leq L_\varepsilon(0) +\varepsilon \int_0^tC_V\norm{r_s}_{\LinftyK}\norm{\eta}_{\infty,C_b(\Kdiag)} \dd s\ .
\end{align*}
Letting $\varepsilon\to 0^+$, Lemma \ref{lemma:quant_int_estimates} implies~\eqref{eq:uniform_boundedness_r}, i.e.,
\begin{align*}
    \norm{r_t}_{\Linfty}&\leq \norm{r_0}_{\Linfty}.
\end{align*}
By repeating the same argument up to~\eqref{eq:differential_inequality_L_eps}, we have
\begin{align*}
    \partial_t L_\varepsilon(t)   &\leq   -L_\varepsilon(t) \int_{\Kepsminus}V_t^+(L_\varepsilon(t),r_t(x'))\eta_t(x,x')\dd\mu(x')
    \\
    & \quad + \int_{\KminusKeps}r_t(x')V_t^-(L_\varepsilon(t),r_t(x'))\eta_t(x,x') \dd\mu(x')\ 
    \\
    &\leq   \ -L_\varepsilon(t) \eta_*\int_{\Kepsminus}V_t^+(L_\varepsilon(t),r_t(x'))\dd\mu(x')\\
    & \quad + C_V\norm{r_t}_{\LinftyK}\norm{\eta}_{\infty,C_b(\Kdiag)}\varepsilon \ ,
\end{align*}
where we used the boundedness of $V$ and $\eta$ and that $\mu(\KminusKeps) < \varepsilon$. The monotonicity of $V$, for $x \in \mathbb{V}_\varepsilon(t)$ and $\mu$-a.e. $x' \in K^\varepsilon$, 
\[
V^+_t(L_\varepsilon(t), r_t(x')) = V_t(L_\varepsilon(t), r_t(x')) \ . 
\]
Furthermore, for velocities satisfying Assumption \ref{ass:V_differentiable}, the mean-value theorem gives
\begin{align*}
    \alpha(L_\varepsilon(t)) - \alpha(r_t(x')) & = \alpha'(c_{t,x})(L_\varepsilon(t) - r_t(x'))
    \\
    & \geq {\alpha'_{*,0}}(L_\varepsilon(t) - r_t(x'))\ ,
\end{align*}
for some $c_{t,x} \in [r_t(x'),L_\varepsilon(t)]\subset \Omega_0:=[0,\|r_0\|_{L^\infty_\mu(K)}]$, where $\alpha'_{*,0}: = \inf_{y\in \Omega_0}\alpha(y)$. Hence,
\begin{align*}
    \partial_t L_\varepsilon(t) \leq & \ -L_\varepsilon(t) \eta_*\alpha'_{*,0}\int_{\Kepsminus}(L_\varepsilon(t) - r_t(x'))\dd\mu(x')
    \\
    & \ + C_V\norm{r_t}_{\LinftyK}\norm{\eta}_{\infty,C_b(\Kdiag)}\varepsilon \ ,
    \\
     = & \  -L^2_\varepsilon(t) \eta_*\alpha'_{*,0}\mu(K^\varepsilon) + \eta_*\alpha'_{*,0} L_\varepsilon(t)\int_{\Kepsminus}r_t(x')\dd\mu(x')
    \\
    & \ + C_V\norm{r_t}_{\LinftyK}\norm{\eta}_{\infty,C_b(\Kdiag)}\varepsilon \ ,
\end{align*}
whence, integrating with respect to time yields
\begin{align*}
     L_\varepsilon(t) &\leq L_\varepsilon(0) -\int_0^t \bigg[L^2_\varepsilon(s) \eta_*\alpha'_{*,0}\mu(K^\varepsilon) - \eta_*\alpha'_{*,0} L_\varepsilon(s)\int_{\Kepsminus}r_s(x')\dd\mu(x')\bigg]\dd s
    \\
    & \quad + \int_0^tC_V\norm{r_s}_{\LinftyK}\norm{\eta}_{\infty,C_b(\Kdiag)}\varepsilon \dd s\ .
\end{align*}
Letting $\varepsilon\to 0^+$, Lemma \ref{lemma:quant_int_estimates} implies 
\begin{align}\label{eq:bernoulli_eq_1}
    \norm{r_t}_{\Linfty}&\leq \norm{r_0}_{\Linfty} -\int_0^t \bigg[\norm{r_s}^2_{\Linfty} \eta_*\alpha'_{*,0}\mu(K) - \eta_*\alpha'_{*,0} \norm{r_s}_{\Linfty}M\bigg] \dd s\ ,
\end{align}
where we have used that $\int_{K}r_t(x)\dd\mu(x') = M$, by the conservation of mass property (see~\cite[Lemma A.0.2]{mikol2024a}).  
 Noting that \eqref{eq:bernoulli_eq_1} is an inequality version of a Bernoulli ODE (see~\cite[Lemma A.0.3]{mikol2024a}), using Lemma \ref{lemma:generalized_gronwall} we infer 
\begin{equation*}
    \norm{r_t}_{\LinftyK} \leq \frac{\norm{r_0}_{\LinftyK}\alpha'_{*,0}\eta_*Me^{\alpha'_{*,0}\eta_*Mt}}{\alpha'_{*,0}\eta_*\left[\norm{r_0}_{\LinftyK}\mu(K)(e^{\alpha'_{*,0}\eta_*Mt} - 1) + M\right]}\ ,
\end{equation*}
which gives the result. 
\end{proof}
Next, we find a lower bound for $r_{t,*}:= \inf_{x \in K} r_{t}(x)$. Note that since $r$ is lower bounded and $K$ is compact, the infimum and the minimum coincide.

\begin{theorem}\label{thm:rate_of_inf}
    Fix $\Phi \equiv\Phi_{upwind}$. Let $(r,\eta)\in AC(\R^+,\LinftyK) \times AC(\R^+,C_b(\Kdiag))$ be a solution of \eqref{eq:euler} with a pointwise monotonic velocity field $V:\R^+\times \LinftyK \to \V^{as}(\Kdiag)$ satisfying \eqref{eq:infty_V_bound}, \eqref{eq:infty_velocity_Lipschitz_second_moment}, and Assumption \ref{ass:V_differentiable}. Let $\w:\R^+\times\LinftyK\times \Kdiag \to \R$ satisfy  \ref{ass:infty_w_continuous}-\ref{ass:infty_omega_symmetric}.    
    Assume $r_0 \geq 0$ and $\int_{K}r_0(x)\dd\mu(x) = M$ and $\eta^0 \in C_b(\Kdiag)$ is positive and symmetric. Then we have 

\begin{align}\label{eq:min_lower_bound}
   r_{t,*} \geq \frac{r_{0,*}\eta_*\alpha'_{*,0}Me^{\eta_*\alpha'_{*,0}M t}}{\eta_*\alpha'_{*,0}(M+ (e^{\eta_*\alpha'_{*,0}M t}-1)\mu(K)r_{0,*})} \ .
\end{align}
\end{theorem}
\begin{proof}
As in Theorem \ref{thm:rate_of_sup_norm}, by our assumptions on $\w$, $r^0,\eta^0$ and the use of the upwind interpolation, we have non-negativity preservation of $r$ for $t\geq 0$ by \cite[Proposition 5.2]{Esposito_on_a_class}. Furthermore, let us consider the Lipschitz continuous reparametrization of the absolutely continuous curve $r_t$. By abuse of notation, let us denote $r_t^- := -r_t$ and fix $\varepsilon>0$. By Theorem \ref{thm:scorza} there exists a compact set $K^\varepsilon\subset K$ such the restriction $r^-:[0,T]\times K^\varepsilon\to \R$ is continuous and $\mu(\KminusKeps) < \varepsilon$ for some $\varepsilon>0$. For all $t \geq 0$ define the family of restricted maxima on $K^\varepsilon$, i.e. 
    \[
    L^-_\varepsilon(t) := \max_{x \in K^\varepsilon} r^-_t(x) \ .
    \]
     Since $L^-_{\varepsilon}(\cdot)$ is the pointwise maximum of equi-Lipschitz continuous functions, Theorem \ref{thm:danskin} implies
    \begin{align}\label{eq:max_min}
        \partial_t L^-_\varepsilon(t)  & = \partial_t\max_{x \in \mathbb{V}^-_\varepsilon(t)}  r_t^-(x) = \max_{x \in \mathbb{V}^-_\varepsilon(t)}\partial_t  r_t^-(x)= \max_{x \in \mathbb{V}^-_\varepsilon(t)} -\partial_t r_t(x)\ ,
    \end{align}
    where $\mathbb{V}^-_\varepsilon(t) = \argmax_{x \in K^\varepsilon} r^-_t(x)$. Then, for any vertex $x \in \mathbb{V}^-_\varepsilon(t)$ it holds that
        \begin{align*}
        \partial_t r^-_t(x)  = &\ r_t(x) \int_{\Kepsminus}V_t^+(r_t(x),r_t(x'))\eta_t(x,x')\dd\mu(x') 
    \\
    &  +r_t(x) \int_{\KminusKeps}V_t^+(r_t(x),r_t(x'))\eta_t(x,x')\dd\mu(x') 
    \\
    & - \int_{\Kepsminus}r_t(x')V_t^-(r_t(x),r_t(x'))\eta_t(x,x') \dd\mu(x')
    \\
    &  - \int_{\KminusKeps}r_t(x')V_t^-(r_t(x),r_t(x'))\eta_t(x,x') \dd\mu(x')\ .
    \end{align*}
Note that $r^\varepsilon_{t,*}:= \min_{x \in K^\varepsilon} r_t(x)= - \max_{x \in K^\varepsilon} r^-_t(x)$. Then, in view of \eqref{eq:max_min} we have for $x \in \mathbb{V}_{\varepsilon}^-(t)$ that
\begin{align}\label{eq:evolution_minimum}
    \partial_t r_{t}(x)  = &\ - r_t(x) \int_{\Kepsminus}V_t^+(r_t(x),r_t(x'))\eta_t(x,x')\dd\mu(x')  \nonumber
    \\
    &\  -r_t(x) \int_{\KminusKeps}V_t^+(r_t(x),r_t(x'))\eta_t(x,x')\dd\mu(x') 
    \\
    &\ + \int_{\Kepsminus}r_t(x')V_t^-(r_t(x),r_t(x'))\eta_t(x,x') \dd\mu(x') \nonumber
    \\
    & \ + \int_{\KminusKeps}r_t(x')V_t^-(r_t(x),r_t(x'))\eta_t(x,x') \dd\mu(x')\ \nonumber
\end{align}
which is the evolution of $ r_t(x)$ for $x \in \argmin_{x\in K^\varepsilon} r_t(x)$, since for any function $f$ such that the maximum and the minimum exist, $\argmax -f =\argmin f$. Furthermore, this implies that 
\begin{align}\label{eq:velocity_property}
    V_t^+(r_t(x), r_t(x')) = \bigg(\alpha(r_t(x)) - \alpha(r_t(x'))\bigg)_+ = 0\ ,
\end{align}
for $x \in \mathbb{V}^-_\varepsilon(t)$ and $x'\in \Kepsminus$, since $\mathbb{V}^-_\varepsilon(t) = \argmax_{x \in K^\varepsilon} -r_t(x) = \argmin_{x \in K^\varepsilon} r_t(x)$ and $\alpha$ is monotonic. In particular, the first term vanishes. Furthermore, noting that the last term is non-negative, we can drop it to obtain the following lower bound
\begin{align*}
    \partial_t r_{t}(x)  \geq &\ - r_{t}(x) \int_{\KminusKeps}V_t^+(r_t(x),r_t(x'))\eta_t(x,x')\dd\mu(x')
    \\
    &  + \int_{\Kepsminus}r_t(x')V_t^-(r_t(x),r_t(x'))\eta_t(x,x') \dd\mu(x')
    \\
     \geq &\ -\norm{r_t}_{\LinftyK}C_V\norm{\eta}_{\infty,C_b(\Kdiag)}\varepsilon
     \\
    & + \int_{\Kepsminus}r_t(x')\bigg(\alpha(r_t(x'))-\alpha(r_t(x))\bigg)_+\eta_t(x,x') \dd\mu(x')
    \\
     = &\ -\norm{r_t}_{\LinftyK}C_V\norm{\eta}_{\infty,C_b(\Kdiag)}\varepsilon
      \\
    & + \int_{\Kepsminus}r_t(x')(\alpha(r_t(x'))-\alpha(r_t(x)))\eta_t(x,x') \dd\mu(x')
     \\
     \geq &\ -\norm{r_t}_{\LinftyK}C_V\norm{\eta}_{\infty,C_b(\Kdiag)}\varepsilon
      \\
    & +\eta_*\alpha'_{*,0}\int_{\Kepsminus} r_t(x')(r_t(x')-r_t(x))\dd\mu(x')
\end{align*}
where in the second and third inequality we used that $x \in \mathbb{V}_\varepsilon^-(t)= \argmin_{x \in K^\varepsilon}r_t(x)$, and in the third inequality we also used the monotonicity of the velocity. Since for $x \in \mathbb{V}^-_\varepsilon(t)$, $r_t(x) = r^\varepsilon_{t,*}$ we have, 
\begin{align*}\label{eq:bernoull_eq_2}
    \partial_t r^\varepsilon_{t,*}  = &  -\norm{r_t}_{\LinftyK}C_V\norm{\eta}_{\infty,C_b(\Kdiag)}\varepsilon
      \\
    & +\eta_*\alpha'_{*,0}\int_{K^\varepsilon} r_t(x')(r_t(x')-r^\varepsilon_{t,*})\dd\mu(x')
    \\
     \geq & -\norm{r_t}_{\LinftyK}C_V\norm{\eta}_{\infty,C_b(\Kdiag)}\varepsilon
      \\
    & +\eta_*\alpha'_{*,0}\left(r^\varepsilon_{t,*}\int_{K^\varepsilon}r_{t}(x') \dd \mu(x') - (r^\varepsilon_{t,*})^2\mu(K^\varepsilon)\right)\ .
\end{align*}
 Integrating in time, sending $\varepsilon \to 0^+$ and using Lemma \ref{lemma:quant_int_estimates} we obtain
\begin{equation}\label{eq:bernoull_eq_2}
    r_{t,*} \geq \ r_{0,*} +\int_0^t \eta_*\alpha'_{*,0}( r_{s,*}M - (r_{s,*})^2\mu(K))\dd s ,
\end{equation}
where we used that $\int_{K}r_t(x)\dd\mu(x) = M$ and the preservation of  mass property. Noting that~\eqref{eq:bernoull_eq_2} resembles a Bernoulli ODE, Lemma \ref{lemma:lower_diff_ineq} and \cite[Lemma A.0.3]{mikol2024a} imply 
\begin{align*}
   r_{t,*} \geq \frac{r_{0,*}\eta_*\alpha'_{*,0}Me^{\eta_*\alpha'_*M t}}{\eta_*\alpha'_{*,0}(M+ (e^{\eta_*\alpha'_{*,0}M t}-1)\mu(K)r_{0,*})} \ .
\end{align*} 
\end{proof}

Now we are ready to obtain the long-time behaviour of the upwind \eqref{eq:euler} for pointwise monotonic velocities. 

\begin{theorem}\label{thm:consensus_convergence}
     Fix $\Phi \equiv\Phi_{upwind}$. Let $(r,\eta)\in AC(\R^+,\LinftyK) \times AC(\R^+,C_b(\Kdiag))$ be a solution of \eqref{eq:euler} with a pointwise monotonic velocity field $V:\R^+\times \LinftyK \to \V^{as}(\Kdiag)$ satisfying \eqref{eq:infty_V_bound}, \eqref{eq:infty_velocity_Lipschitz_second_moment} and Assumption \ref{ass:V_differentiable}. Let $\w:\R^+\times\LinftyK\times \Kdiag \to \R$ satisfy \ref{ass:infty_w_continuous}-\ref{ass:infty_omega_symmetric}.    
    Assume $r_0 \geq 0$ and $\int_{K}r_0(x)\dd\mu(x) = M$, and $\eta^0 \in C_b(\Kdiag)$ is positive and symmetric. Then, for $\mu$-a.e. $x \in \Rd$, it holds
\[
\lim_{t \to \infty}r_t(x) = \frac{M}{\mu(K)},
\]
the uniform distribution of mass over $K$. 
\end{theorem}
\begin{proof}
    For $\mu$-a.e. $x\in\Rd$ and $t\geq 0$ we have, by~\eqref{eq:max_upper_bound} and \eqref{eq:min_lower_bound}, that
    \begin{align*}
        \frac{r_{0,*}\eta_*\alpha'_{*,0}Me^{\eta_*\alpha'_{*,0}M t}}{\eta_*\alpha'_{*,0}(M+ (e^{\eta_*\alpha'_{*,0}M t}-1)\mu(K)r_{0,*})}  
       &\leq  r_{t,*}
       \\
       & \leq r_t(x) 
       \\
       & \leq \norm{{r_t}}_{{\LinftyK}}\
        \\
        & \leq \frac{\norm{r_0}_{\LinftyK}\alpha'_{*,0}\eta_*Me^{\alpha'_{*,0}\eta_*Mt}}{\alpha'_{*,0}\eta_*\left[\norm{r_0}_{\LinftyK}\mu(K)(e^{\alpha'_{*,0}\eta_*Mt} - 1) + M\right]}\ ,
    \end{align*}
    and letting $t \to \infty$ yields the result.     
\end{proof}

As a direct consequence of the previous theorem and the results in~Section~\ref{sec:graph_CE} we have the following theorem for the long-time asymptotics of the $\mu$-monokinetic solutions to~\eqref{eq:vlasov_equation}.
\begin{corollary}
Let $\Phi \equiv\Phi_{upwind}$ and let $(r,\eta)\in AC(\R^+,\LinftyK) \times AC(\R^+,C_b(\Kdiag))$ be a solution of \eqref{eq:euler} with a pointwise monotonic velocity field $V:\R^+\times \LinftyK \to \V^{as}(\Kdiag)$ satisfying \eqref{eq:infty_V_bound}, \eqref{eq:infty_velocity_Lipschitz_second_moment} and Assumption \ref{ass:V_differentiable}. Let $\w:\R^+\times\LinftyK\times \Kdiag \to \R$ satisfy \ref{ass:infty_w_continuous}-\ref{ass:infty_omega_symmetric}.    
    Assume $r_0 \geq 0$ and $\int_{K}r_0(x)\dd\mu(x) = M$, and $\eta^0 \in C_b(\Kdiag)$ is positive and symmetric. Then, the $\mu$-monokinetic solution to~\eqref{eq:vlasov_equation} weakly-* converges to $\bar\sigma=\delta_{M/\mu(K)}\otimes\mu$, as $t\to\infty$.
\end{corollary}
\begin{proof}
The $\mu$-monokinetic solution to~\eqref{eq:vlasov_equation} is given by $\sigma_t=\delta_{r_t}\otimes\mu$. One can check, indeed, that $\sigma_t$ satisfies~\eqref{eq:mu_ae_sol_concept_vlasov} for $\varphi\in C_c^1([0,+\infty)\times \R\times K)$. Due to~Theorem~\ref{thm:consensus_convergence}, for any $f\in C_0(\R\times K)$ it holds
\[
\lim_{t\to\infty}\int_{\R\times K}f(\xi,x)\dd \delta_{r_t(x)}\otimes\mu(\xi,x)=\lim_{t\to\infty}\int_{K}f(r_t(x),x)\dd \mu(x)=\int_{K}f(M/\mu(K),x)\dd \mu(x),
\]
whence the thesis.
\end{proof}

\subsection*{Acknowledgements}
The authors are grateful to Ben Hambly and Markus Schmidtchen for their valuable comments and insightful remarks. The authors were supported by the Advanced Grant Nonlocal-CPD (Nonlocal PDEs for Complex Particle Dynamics: Phase Transitions, Patterns and Synchronization) of the European Research Council Executive Agency (ERC) under the European Union’s Horizon 2020 research and innovation programme (grant agreement No. 883363). AE acknowledges partial support by the EPSRC grant number EP/T022132/1 and by the InterMaths Network (www.intermaths.eu). LM was mainly supported by the EPSRC Centre for Doctoral Training in Mathematics of Random Systems: Analysis, Modelling and Simulation (EP/S023925/1).

\subsection*{Availability of data and materials}

Data sharing not applicable to this article as no datasets were generated or analysed during the current study.

\bibliographystyle{abbrv}
\bibliography{biblio}

\end{document}